\documentclass[a4paper]{amsart}
\usepackage{amssymb, enumitem}
\usepackage[all]{xy}
\usepackage{hyperref, aliascnt}
\usepackage{mathtools}
\usepackage[english]{babel}
\def\today{\number\day\space\ifcase\month\or   January\or February\or
   March\or April\or May\or June\or   July\or August\or September\or
   October\or November\or December\fi\   \number\year}

\theoremstyle{definition}

\newaliascnt{thmCt}{lma}
\newtheorem{thm}[thmCt]{Theorem}
\aliascntresetthe{thmCt}

\newaliascnt{corCt}{lma}
\newtheorem{cor}[corCt]{Corollary}
\aliascntresetthe{corCt}

\newaliascnt{cnjCt}{lma}

\aliascntresetthe{cnjCt}

\newaliascnt{propCt}{lma}
\newtheorem{prop}[propCt]{Proposition}
\aliascntresetthe{propCt}

\newtheorem*{thm*}{Theorem}
\newtheorem*{cor*}{Corollary}
\newtheorem*{prop*}{Proposition}
\newtheorem*{qst*}{Question}

\newcounter{theoremintro}
\newtheorem{thmintro}[theoremintro]{Theorem}

\newtheorem{egintro}[theoremintro]{Example}

\newaliascnt{pgrCt}{lma}

\aliascntresetthe{pgrCt}

\newaliascnt{dfCt}{lma}
\newtheorem{df}[dfCt]{Definition}
\aliascntresetthe{dfCt}

\newaliascnt{remCt}{lma}
\newtheorem{rem}[remCt]{Remark}
\aliascntresetthe{remCt}

\newaliascnt{remsCt}{lma}

\aliascntresetthe{remsCt}

\newaliascnt{egCt}{lma}
\newtheorem{eg}[egCt]{Example}
\aliascntresetthe{egCt}

\newaliascnt{exCt}{lma}

\aliascntresetthe{exCt}

\newaliascnt{egsCt}{lma}

\aliascntresetthe{egsCt}

\newaliascnt{qstCt}{lma}

\aliascntresetthe{qstCt}

\newaliascnt{pbmCt}{lma}

\aliascntresetthe{pbmCt}

\newaliascnt{notaCt}{lma}
\newtheorem{nota}[notaCt]{Notation}
\aliascntresetthe{notaCt}

\newcommand{\beq}{\begin{equation}}
\newcommand{\eeq}{\end{equation}}
\newcommand{\beqa}{\begin{eqnarray*}}
\newcommand{\eeqa}{\end{eqnarray*}}
\newcommand{\bal}{\begin{align*}}
\newcommand{\eal}{\end{align*}}
\newcommand{\bi}{\begin{itemize}}
\newcommand{\ei}{\end{itemize}}
\newcommand{\be}{\begin{enumerate}}
\newcommand{\ee}{\end{enumerate}}

\newcommand{\ep}{\varepsilon}
\newcommand{\zt}{\zeta}

\newcommand{\Z}{{\mathbb{Z}}}
\newcommand{\C}{{\mathbb{C}}}
\newcommand{\N}{{\mathbb{N}}}

\newcommand{\K}{{\mathcal{K}}}

\newcommand{\B}{{\mathcal{B}}}
\newcommand{\U}{{\mathcal{U}}}

\newcommand{\T}{{\mathbb{T}}}

\newcommand{\OI}{{\mathcal{O}_{\I}}}

\newcommand{\id}{{\mathrm{id}}}

\newcommand{\supp}{{\mathrm{supp}}}

\newcommand{\Aut}{{\mathrm{Aut}}}

\newcommand{\Hom}{{\mathrm{Hom}}}
\newcommand{\Ext}{{\mathrm{Ext}}}

\newcommand{\Ad}{{\mathrm{Ad}}}

\newcommand{\ifo}{if and only if }

\newcommand{\ca}{$C^*$-algebra}

\newcommand{\uca}{unital $C^*$-algebra}

\newcommand{\Rp}{Rokhlin property}

\newcommand{\cRp}{continuous Rokhlin property}

\newcommand{\I}{\infty}

\title{Equivariant 
KK-theory and the continuous Rokhlin property}

\date{\today}

\thanks{This material is based upon work supported by the
  US National Science Foundation through my thesis advisor's Grant
DMS-1101742, and by the the Deutsche Forschungsgemeinschaft
(SFB 878). Both sources of financial support are
gratefully acknowledged.}

\author[Eusebio Gardella]{Eusebio Gardella}
\address{Eusebio Gardella
Mathematisches Institut, Fachbereich Mathematik und Informatik der
Universit\"at M\"unster, Einsteinstrasse 62, 48149 M\"unster, Germany.}
\email{gardella@uni-muenster.de}
\urladdr{www.math.uni-muenster.de/u/gardella/}


\begin{document}

\begin{abstract}
We introduce and study the continuous Rokhlin property for actions of 
compact groups on $C^*$-algebras. 
An important technical result is a characterization of the continuous Rokhlin property in terms of asymptotic retracts. As a consequence, we derive strong $KK$-theoretical obstructions to the
continuous Rokhlin property. Using these, we show that the UCT is preserved under formation of crossed
products and passage to fixed point algebras by such actions, even in the absence of nuclearity. 

Our analysis of the $KK$-theory of the crossed product allows us to prove
a $\T$-equivariant version of Kirchberg-Phillips: 
two circle actions with the continuous Rokhlin property on Kirchberg algebras are conjugate 
whenever they are $KK^{\T}$-equivalent. In the presence of the UCT, this is equivalent to having isomorphic equivariant 
$K$-theory. We moreover characterize the $KK^{\T}$-theoretical invariants that arise in this way.

Finally, we identify a $KK^{\T}$-theoretic obstruction to the continuous property, which is shown to be the only obstruction in the setting of 
Kirchberg algebras.
We show by means of explicit examples that the Rokhlin property is strictly weaker than the continuous Rokhlin property. 
\end{abstract} 

\maketitle

%\tableofcontents

\renewcommand*{\thetheoremintro}{\Alph{theoremintro}}
\section*{Introduction}

A celebrated result in ergodic theory is the Rokhlin lemma, which asserts that an
aperiodic, measure preserving ergodic transformation can be approximated
by finite cyclic shifts. In terms of operator algebras, this
is equivalent to writing the constant function 1 as a finite sum of projections
on which the automorphism acts via cyclic permutations.
Its wide applicability in measurable
dynamics motivated the search for noncommutative versions of the Rokhlin lemma.
For properly outer automorphisms of finite von Neumann algebras, this was
achieved by Connes \cite{Con_classification_1976}, and generalizations 
to actions of amenable groups on the
hyperfinite II$_1$-factor were obtained by Jones and Ocneanu. 

The Rokhlin property has also been extremely fruitful in the context of \ca s. Early
works include the studies of cyclic group actions on UHF-algebras by
%Herman and Jones \cite{HerJon_period_1982}, and 
Herman and Ocneanu \cite{HerOcn_spectral_1986}, and later
for automorphisms by Kishimoto \cite{Kis_rohlin_1995}. Izumi's praised results \cite{Izu_finiteI_2004}
and \cite{Izu_finiteII_2004} on finite
group actions on unital \ca s represent a cornerstone in the recent literature. (See \cite{GarSan_equivariant_2016}
for the non-unital case.)
%, and motivated a number of subsequent works.

%In \cite{OsaPhi_crossed_2012}, Osaka and Phillips showed that crossed products and fixed point algebras by finite group actions with the \Rp\ inherit a
%number of structural properties of the underlying algebra, using what they call ``flexible classes'' of \ca s.
The \Rp\ allows one to take averages over the group in such
a way that $\ast$-algebraic identities are approximately respected. 
%It is in this way that we think of the Rokhlin property as the technical condition
%that makes it possible to average over the group while preserving, at least up to some prescribed error, algebraic
%relations. 
When this averaging technique
is combined with some form of stability, one can show that a number of properties of the original
algebra pass to the fixed point algebra; see \cite{OsaPhi_crossed_2012, Gar_crossed_2014, Gar_crossed_2019}.

In this paper, we study this averaging process in more depth, under the assumption that the action have what we shall call the
\emph{continuous} Rokhlin property. 
%We focus on circle actions because we only prove classification theorems for these, but many
%of our results hold for arbitrary compact groups. When appropriate, we describe how to modify our arguments to deal with
%general compact groups.
The \cRp\ can be characterized in terms of asymptotic retracts:

\begin{thmintro}
(See \autoref{thm: existence of a.m.}). Let
$\alpha\colon G\to\Aut(A)$ be an action of a compact group $G$
on a \uca\ $A$. Then
$\alpha$ has the \cRp\ if and only if there exists a continuous path $(\Psi_t)_{t\in [0,\I)}$
of unital completely positive asymptotically multiplicative and 
asymptotically equivariant maps $\Psi_t\colon C(G)\otimes A\to A$ 
satisfying
$\lim\limits_{t\to\I}\Psi_t(1\otimes a)=a$ for all $a\in A$.
\end{thmintro}

An important consequence of Theorem~A is the fact that actions with the continuous Rokhlin property
preserve the UCT, even in the absence of nuclearity:

\begin{thmintro}(See \autoref{thm: UCTcRp}).
Let $G$ be a compact group and let $\alpha\colon G\to\Aut(A)$ 
be an action with the \cRp.
If $A$ satisfies the UCT, then so do $A\rtimes_\alpha G$ and $A^\alpha$.
\end{thmintro}

We point out that we do not work in the nuclear setting. This makes our arguments, particularly in the proof
of Theorem~A, necessarily more technical, but has the advantage of being completely general. 
%As an application of Theorem~B (for finite groups), 
%we answer a question of Phillips-Viola by showing that their example
%\cite{PhiVio_simple_2013} of a simple, exact \ca\ which 
%is not isomorphic to its opposite, also satisfies the UCT; see \autoref{thm:PhiVioUCT}.

Using the picture of $KK$-theory in terms of completely positive contractive asymptotic homomorphisms by Hough\-ton-Larsen and Thomsen \cite{HouTho_universal_1999}, it follows that, in the context of
Theorem~A, the group $KK(A^\alpha,B)$ is a direct summand of $KK(A,B)$ for any \ca\ $B$; see \autoref{cor: E-thy summand}. 
For circle actions, an analysis of the dual automorphism
of an action with the \cRp\ (see \autoref{thm:duality}) allows
us to prove an even stronger result. 
We say that two actions
$\alpha\colon\T\to\Aut(A)$ and $\beta\colon \T\to\Aut(B)$
on unital \ca s $A$ and $B$ are \emph{unitally $KK^\T$-equivalent},
if there exists an invertible element $\eta\in KK^\T(A,B)$
with $[1_A]\times \eta=[1_B]$ in $K_0^\beta(B)$.

%We say that a \ca\ $A$ is \emph{$KK$-symmetric} if there exists a 
%\ca\ $B$ such that $A\sim_{KK}B\oplus SB$.

\begin{thmintro}(See \autoref{prop:KKTequivalenceCTAalpha}).
Let $\alpha\colon \T\to\Aut(A)$ be an action with the continuous \Rp\ on a unital, separable \ca\ $A$.
Then there is a unital $KK^\T$-equivalence 
$A\sim_{KK^\T}C(\T)\otimes A^\alpha$. 
In particular, $K_0(A)\cong K_1(A)\cong K_0(A^\alpha)\oplus K_1(A^\alpha)$.
\end{thmintro}

In Section~3 we obtain $\T$-equivariant versions of a number of 
celebrated results about Kirchberg algebras, mainly due to Kirchberg 
and Kirchberg-Phillips. 
The first one in this direction provides a complete classification of 
circle actions with the \cRp\ on Kirchberg algebras in terms of
$KK^\T$-theory. 

%, in combination with a result of Nakamura from \cite{Nak_aperiodic_2000}, to obtain a classification
%result for such actions on Kirchberg algebras (purely infinite, simple,
%separable, unital, nuclear $C^*$-algebras),
%in terms of $KK$-theory (or $K$-theory in the UCT case). We can also completely
%describe the range of the invariant. These results can be summarized as follows:

\begin{thmintro}(See Theorems~\ref{thm: ClassCRp} and \ref{thm: rangeInvCRp}).
Let $A$ be a unital Kirchberg algebra.
\be
\item There exists an action of $\T$ on $A$ with the \cRp\ if and only if $A$ is nuclearly $KK$-symmetric (see \autoref{df:KKsymmetric}). 
In the presence of the UCT, this is equivalent to $K_0(A)\cong K_1(A)$.
\item Let $\alpha\colon\T\to\Aut(A)$ and
$\beta\colon\T\to\Aut(B)$ be actions with the continuous Rokhlin property, with $B$ a unital Kirchberg algebra. 
Then $(A,\alpha)$ and $(B,\beta)$
are conjugate if and only if they are unitally $KK^\T$-equivalent. 
In the presence of the UCT, this is equivalent to 
$(A,\alpha)$ and $(B,\beta)$ having isomorphic $\T$-equivariant $K$-theory. 
%the existence of an isomorphism
%\[(K^\T_0(A),[1_{A}],K^\T_1(A))\cong (K^\T_0(B),[1_{B}],K^\T_1(B)).\]
\item Suppose that $A$ satisfies the UCT and that 
$K_0(A)\cong K_1(A)$. 
Then a triple $(H_0,h_0,H_1)$ consisting of two abelian groups $H_0$ and $H_1$, and a distinguished
element $h_0\in H_0$, arises as the equivariant $K$-theory 
of a circle action on $A$ with
the \cRp\ if and only if there exists an isomorphism $\varphi\colon H_0\oplus H_1\to K_0(A)$ with $\varphi(h_0,0)=[1_A]$.
\ee
\end{thmintro}

It follows from parts~(2) and~(3) of Theorem~D that there exists, up to conjugacy, a unique action of $\T$ on $\mathcal{O}_2$
with the continuous \Rp.

It is well-known that every unital, nuclear, separable 
\ca\ is unitally $KK$-equivalent to a
(unique) 
unital Kirchberg algebra; see 
Proposition~8.4.5 in~\cite{RorSto_classification_2002} and 
\autoref{rem:NucSepKKeqKirch}.
For circle actions with the \cRp, the equivariant version of this result
is also true:

\begin{thmintro}(See \autoref{thm:EverycRpKKequivKirch}).
Let $\alpha\colon \T\to\Aut(A)$ be an action with the \cRp\
on a separable, unital, nuclear \ca\ $A$. Then there
exist a (unique) unital Kirchberg algebra $B$ and a (unique) 
action $\beta\colon \T\to\Aut(B)$
with the \cRp\ such that $(A,\alpha)$ is unitally 
$KK^\T$-equivalent to $(B,\beta)$.
\end{thmintro}

In the context of Kirchberg algebras, the difference between the continuous Rokhlin property and the \Rp\ amounts to the difference between the functors $KL$ and $KK$. Indeed, in this setting we identify 
precisely what the obstruction to the continuous Rokhlin property is:

\begin{thmintro}
Let $\alpha\colon\T\to\Aut(A)$ be an action with the \Rp\ on
a unital Kirchberg algebra $A$. Then $\alpha$
has the \cRp\ if and only if $KK(\widehat{\alpha})=1$ in $KK(A\rtimes_\alpha \T, A\rtimes_\alpha\T)$.
\end{thmintro}

The condition in the theorem above is not automatic, and we produce an explicit example
showing that there is a legitimate difference between the \Rp\ and the \cRp.

\begin{egintro}
There exist a unital Kirchberg algebra satisfying the UCT and a circle action $\alpha\colon\T\to\Aut(A)$ with the \Rp\ which
does not have the \cRp.
\end{egintro}

%\vspace{0.3cm}

%The paper is organized as follows. In Section 1, we
%introduce the definition of the continuous Rokhlin property for a circle action and develop its basic theory.
%Our definition of the continuous \Rp\
%is a strengthening of the \Rp\ of Hirshberg and Winter from \cite{HirWin_rokhlin_2007}, which asks for a continuous path of unitaries rather than a sequence. 

%Most of the hard work in this paper is done in Section~2, where we study the $KK$-theory of circle actions with the \cRp\ and prove
%Theorems~A,~B, and~C. 
%Section~3 contains the proof of Theorem~D.

%Finally, in Section~4 we find a $KK^{\T}$-theoretic obstruction to the 
%continuous Rokhlin property, and show that this is the only obstruction
%for Rokhlin actions on Kirchberg algebras. It follows that the two 
%notions are equivalent for actions on Kirchberg algebras with finitely
%generated $K$-theory, but differ in general by Example~F.

%\vspace{0.3cm}

\textbf{Acknowledgements.} Most of this work was done while the author was visiting the WWU M\"unster in
2013, and while the author was participating in the Thematic Program on Abstract Harmonic Analysis, Banach and Operator Algebras, at the
Fields Institute in 2014. The author thanks both Mathematics
Institutes for their hospitality and for providing a stimulating research environment.

The author is very grateful to Chris Phillips for several helpful conversations.
He would also like to express his gratitude to 
Selcuk Barlak, Martino Lupini, Ralf Meyer, Luis Santiago, and Hannes Thiel for several helpful discussions. He also thanks Masaki Izumi for sharing
some unpublished work of his (\cite{Izu_preparation_2014}), 
and Shirly Geffen for reading a preliminary version of this 
work.
Finally, the author thanks both anonymous referees for their
thorough reports, and in particular for suggesting a simpler
proof of \autoref{thm: existence of a.m.}.

\vspace{0.3cm}

\textbf{Notation and terminology.}
For a \ca\ $A$, we denote by $\Aut(A)$ the automorphism group of $A$, 
and by $SA$ its suspension.
If $A$ is moreover unital, then $\U(A)$ denotes the unitary group of $A$. We set $\N=\{1,2,\ldots\}$.

If $\alpha\colon G\to\Aut(A)$ is an action of a locally
compact group $G$ on $A$, then we will denote by $A^\alpha$ its fixed point subalgebra.
%
%For a \ca\ $A$, we set
%\begin{align*} \ell^\I(\N,A)&=\left\{(a_n)_{n\in\N}\in A^{\N}\colon \sup_{n\in\N}\|a_n\|<\I\right\};\\
%c_0(\N,A)&=\left\{(a_n)_{n\in\N}\in\ell^\I(\N,A)\colon \lim\limits_{n\to\I}\|a_n\|= 0\right\},\end{align*}
%and define its \emph{sequence algebra} $A_\I$ by $A_\I= \ell^\I(\N,A)/c_0(\N,A)$. We identify $A$ with the constant sequences in $\ell^\I(\N,A)$ and
%with its image in $A_\I$. We write $A_\I\cap A'$ for the \emph{central sequence algebra} of $A$, that is, the
%relative commutant of $A\in A_\I$. For a sequence $(a_n)_{n\in\N}\in A$, we denote by $[(a_n)_{n\in\N}]$ its image in $A_\I$. \\
%\indent 
%If $\alpha\colon G\to\Aut(A)$ is an action of $G$ on $A$, then there are actions of $G$ on $A_\I$ and on $A_\I\cap A'$, both denoted by $\alpha_\I$. Note that unless
%the group $G$ is discrete, these actions will in general not be continuous. 
When $G$ is abelian, the dual action of $\widehat{G}$ on 
$A\rtimes_\alpha G$ is denoted by $\widehat{\alpha}$. When $G$ is
discrete, this action is determined by
$\widehat{\alpha}_\chi(au_g)=\chi(g)au_g$ for all $\chi\in\widehat{G}$, for all $a\in A$ and for all $g\in G$.
For a compact group $G$, we denote by 
$\verb'Lt'\colon G\to\Aut(C(G))$ the action induced by left translation.

\section{\texorpdfstring{$KK$}{KK}-retracts and the UCT}
In this section, we introduce the definition of the \cRp\ for actions
of compact groups, and show that it is equivalent to the existence of 
a completely positive equivariant asymptotic morphism $C(G,A)
\rightrightarrows A$
which is the identity on $A$; see \autoref{thm: existence of a.m.}. 
The fact that this asymptotic morphism consists of completely positive
maps is crucial, and we obtain a number of consequences of this 
characterization. For example, we relate the $KK$-theory of $A$ to that of $A^\alpha$ in \autoref{cor: E-thy summand}, and we show that 
the UCT for $A$ implies the UCT for $A^\alpha$ and for 
$A\rtimes_\alpha G$; see \autoref{thm: UCTcRp}. For 
compact Lie groups, an independent treatment can be 
found in \cite{AraKub_compact_2017}. 

Some arguments in this section can be simplified if one is only interested in \emph{nuclear}
$C^*$-algebras, by using $E$-theory instead; see
\cite{Sza_short_2015}. 
Our approach, despite being more technical, requires only minimal assumptions and gives new information in cases 
of interest; see, for example, 
\cite{PhiVio_simple_2013}.

We recall the definition of a completely positive contractive asymptotic morphism from \cite{HouTho_universal_1999}.

\begin{df} \label{df: asymptotic morph}
Let $A$ and $B$ be \ca s. An \emph{asymptotic morphism} from $A$ to $B$,
written $\Psi\colon A\rightrightarrows B$, is a family $(\Psi_t)_{t\in [0,\I)}$ of maps $A\to B$, satisfying:
\be\item for every $a\in A$, the map $[0,\I)\to B$ given by $t\mapsto \Psi_t(a)$ is continuous;
\item for every $\lambda\in \C$ and every $a$ and $b\in A$, we have
\begin{align*} \lim\limits_{t\to\I}\|\Psi_t(\lambda a+b)&-\lambda\Psi_t(a)-\Psi_t(b)\|=0, \\
 \lim\limits_{t\to\I}\|\Psi_t(ab)-\Psi_t(a)\Psi_t(b)\|=0,\ \ &\mbox{ and } \ \ \lim\limits_{t\to\I}\|\Psi_t(a^*)-\Psi_t(a)^*\|=0.
\end{align*}\ee

Let $\Psi\colon A\rightrightarrows B$ be an asymptotic morphism.
We say that $\Psi$ is \emph{completely positive} (respectively, \emph{unital}),
if there exists $t_0\in [0,\I)$ such that $\Psi_t$ is completely positive (respectively, unital)
for all $t\geq t_0$. We say that $\Psi$ is \emph{asymptotically central}
if $\lim\limits_{t\to\I}\|\Psi_t(a)b-b\Psi_t(a)\|=0$ for all $a\in A$ and 
all $b\in B$. 
When $A$ and $B$ carry actions $\alpha$ and $\beta$ of some compact
group $G$, we say that $\Psi$ is 
\emph{equivariant} if for all $a\in A$, one has
\[\lim\limits_{t\to\I}\sup_{g\in G}\|\Psi_t(\alpha_g(a))-
\beta_g(\Psi_t(a))\|=0.\] 
\end{df}

%\begin{rem}
%$E$-theory was introduced by Connes and Higson in \cite{ConHig_deformations_1990}, using a suitable equivalence between
%asymptotic morphisms between $C^*$-algebras. Despite coinciding when the first variable is nuclear,
%$E$-theory and $KK$-theory do not in general agree. Even more,
%there are $C^*$-algebras that satisfy the UCT in $E$-theory, but do not satisfy the UCT (in $KK$-theory).%; see
%\cite{Ska_bifoncteur_1991} (we are thankful to Rasmus Bentmann for providing this reference).
%\end{rem}

We begin with the main definition of this work. We give it in a form that
is convenient for our purposes, and note that it can be rephrased in 
terms of central path algebras using the Choi-Effros lifting theorem for 
completely positive maps. For finite groups, this 
definiton has already appeared in Definition~3.1 of
\cite{PhiVio_simple_2013}, and was originally 
introduced by Phillips in~\cite{Phi_preparation_2021}.

\begin{df}\label{df: cRpG}
Let $G$ be a second-countable compact
group, let $A$ be a separable unital \ca, and let $\alpha\colon G\to\Aut(A)$ be a continuous action. We say that $\alpha$ has
the \emph{continuous Rokhlin property} if there exists a unital 
completely positive equivariant and asymptotically central
asymptotic morphism 
$\Phi\colon C(G)\rightrightarrows A$.\end{df} 

This is a strengthening of the Rokhlin property from \cite{HirWin_rokhlin_2007,Izu_finiteI_2004,Gar_crossed_2014, GarLup_applications_2018, Gar_circle_2014}: roughly
speaking, the difference is that sequences are replaced with continuous paths. We reproduce the definition below.

\begin{df}\label{df:Rp}
Let $G$ be a second-countable compact
group, let $A$ be a separable unital \ca, and let $\alpha\colon G\to\Aut(A)$ be a continuous action. We say that $\alpha$ has
the \emph{Rokhlin property} if there exists a 
sequence $(\varphi_n)_{n\in\N}$ of unital 
completely positive maps $\varphi_n \colon C(G) \to A$ satisfying
\bi
\item $\lim\limits_{n\to\I}\|\varphi_n(f_1f_2)-\varphi_n(f_1)(f_2)\|=0$
for all $f_1,f_2\in C(G)$;
\item $\lim\limits_{n\to\I}\|\varphi_n(f)a-a\varphi_n(f)\|=0$ for all $f\in C(G)$ and all $a\in A$;
\item $\lim\limits_{n\to\I}\sup\limits_{g\in G}\|\varphi_n(\alpha_g(f))-
\beta_g(\varphi_n(f))\|$ 
for all $f\in C(G)$.\ei\end{df}

In particular, any action with the \cRp\ has the \Rp. The 
converse is however not true; see \autoref{eg: not cRp 1} and \autoref{eg: not cRp 2}. 

On the other hand, many natural examples of actions 
with the Rokhlin property do have the continuous \Rp. The basic example is the action $\texttt{Lt}\colon G\to\Aut(C(G))$ by left translation, for which one may take 
$\Phi_t=\id_{C(G)}$ for all $t\in [0,\I)$. Another canonical example
is that of Izumi's model action of a finite group:

\begin{eg}\label{eg:ModelActionFiniteG} 
Let $G$ be a finite group, let $\lambda\colon G\to\U(\ell^2(G))$ be the 
left regular representation, and denote by $D_G$ the UHF-algebra 
$D_G=\bigotimes\limits_{n\in\N}\B(\ell^2(G))$. Denote by 
$\mu^G\colon G\to \Aut(D_G)$ the product-type action
$\mu^G_g=\bigotimes\limits_{n\in\N} \Ad(\lambda_g)$ for all $g\in G$. Then 
$\mu^G$ has the \cRp;
see Lemma~3.4 in~\cite{PhiVio_simple_2013}.
\end{eg} 

%The action $\texttt{Lt}$ can be used to produce actions of an arbitrary
%compact group $G$ on simple AH-algebras 
%with the continuous Rokhlin property, by taking diagonal connecting maps.
%The continuous Rokhlin property enjoys preservation proprties similar to 
%the Rokhlin property with respect to tensor products, direct sums, 
%quotients, etc. On the other hand, care is needed to deal with direct limits (in particular, to connect the unitaries from the different 
%finite stages). 

We do not focus on constructing further examples by hand, since large families are
constructed in \autoref{thm:duality} and \autoref{thm: rangeInvCRp}. 
We now give a different formulation of the \cRp\ using
central path algebras.

\begin{nota}\label{nota:CtsPathAlg}
Let $G$ be a compact group.
For a \ca\ $A$ and an action $\alpha\colon G\to\Aut(A)$, 
we denote by $C_{b,\alpha}([0,\I),A)$ the subalgebra of 
$C_b([0,\I),A)$ consisting of those elements where the 
canonical action induced by pointwise application of
$\alpha$ is strongly continuous. We write 
$A_{\mathrm{c},\alpha}$ for the quotient of 
$C_{b,\alpha}([0,\I),A)$ by the $G$-invariant ideal
$C_0([0,\I),A)$. We call $A_{\mathrm{c},\alpha}$ 
the \emph{path algebra} of $A$, and we write
$\pi_A\colon C_{b,\alpha}([0,\I),A)\to A_{\mathrm{c},\alpha}$
for the canonical quotient map. We denote the induced action by
$\alpha_{\mathrm{c}}\colon G\to\Aut(A_{\mathrm{c},\alpha})$.
Note that there is a canonical equivariant embedding
$A\to A_{\mathrm{c},\alpha}$ as (equivalence classes of)
constant functions, and we denote by
$A_{\mathrm{c},\alpha}\cap A'$ the relative commutant.
\end{nota}

\begin{rem}\label{rem:CtsPathAlg}
Path algebras are useful in connection with asymptotic morphisms,
since for a \ca\ $C$ with an acton $\gamma\colon G\to\Aut(C)$,
(unital, completely positive) 
equivariant asymptotic morphisms 
$\Psi\colon C\rightrightarrows A$ are in one-to-one
correspondence with (unital, completely positive) 
maps $\psi\colon C\to 
C_{b,\alpha}([0,\I),A)$ such that $\pi_A\circ\psi$ is an 
equivariant homomorphism: the correspondence is given by
$\psi(t)=\Psi_t$ for all $t\in [0,\I)$. 
Moreover, asymptotic centrality of 
$\Psi$ amounts to the fact that the range of $\pi_A\circ\psi$
is contained in $A_{\mathrm{c},\alpha}\cap A'$.
 \end{rem}

For an action $\alpha\colon G\to\Aut(A)$, we endow $C(G)\otimes A$ with the diagonal action $\texttt{Lt}\otimes\alpha$.
The following is the main technical result of this section.
The implication from (1) to (2) generalizes Proposition~3.7
of~\cite{PhiVio_simple_2013} from finite groups to compact
groups.

\begin{thm}\label{thm: existence of a.m.} Let $A$ be a unital separable \ca\ and let $\alpha\colon G\to\Aut(A)$ be an action of a second countable compact group $G$. Then the following are equivalent:
\be\item the action $\alpha$ has the continuous Rokhlin property; 
%Denote by $\iota\colon A^\alpha\to A$ the canonical inclusion.
\item there exists a unital completely positive equivariant 
asymptotic morphism
$\Upsilon\colon C(G)\otimes A \rightrightarrows A$ satisfying
$\lim\limits_{t\to\I}\Upsilon_t(1\otimes a)=a$ for all $a\in A$.
\ee
In particular, if $\alpha$ has the \cRp, then there exists an 
equivariant
asymptotic morphism $\Psi\colon A\rightrightarrows A^\alpha$ satisfying
$\lim\limits_{t\to\I}\Psi_t(a)=a$ for all $a\in A^\alpha$.
\end{thm}
\begin{proof}
Suppose that an asymptotic morphism $\Upsilon$ as in (2) exists. For 
$t\in [1,\I)$, let 
$\Phi_t\colon C(G)\to A$ be given by 
$\Phi_t(f)=\Upsilon_t(f\otimes 1_A)$ for all $f\in C(G)$. Then 
$\Phi=(\Phi_t)_{t\in [0,\I)}$ is a unital, completely positive
and equivariant asymptotic morphism, and it remains to check asymptotic
centrality. Given $f\in C(G)$ and $a\in A$, we have
\begin{align*}
\lim_{t\to\I}\|\Phi_t(f)a-a\Phi_t(a)\|&=\lim_{t\to\I}\|\Upsilon_t(f\otimes 1)\Upsilon_t(1\otimes a)- \Upsilon_t(1\otimes a)\Upsilon_t(f\otimes 1)\|=0,
\end{align*}
since $\Upsilon_t(f\otimes 1)\Upsilon_t(1\otimes a)$ and $\Upsilon_t(1\otimes a)\Upsilon_t(f\otimes 1)$ are both asymptotically equal to 
$\Upsilon_t(f\otimes a)$. Thus $\alpha$ has the 
\cRp.

We prove the converse.
Using the \cRp\ for $\alpha$ together with \autoref{rem:CtsPathAlg}, find a 
unital, completely positive map 
$\varphi\colon C(G)\to C_{b,\alpha}([0,\I),A)$ such that 
$\pi_A\circ\varphi\colon C(G)\to A_{\mathrm{c},\alpha}\cap A'$ 
is an equivariant homomorphism. 

Set $\phi=\pi_A\circ\varphi$.
Since the range of $\phi$ commutes with the canonical copy
of $A$ in $A_{\mathrm{c},\alpha}$, 
there is a unital, equivariant homomorphism 
$\phi\otimes \id_A\colon C(G,A)\to A_{\mathrm{c},\alpha}$
satisfying $(\phi\otimes\id_A)(1\otimes a)=a$
for all $a\in A$. By \autoref{rem:CtsPathAlg}, it suffices
to show that there is a unital, completely positive
lift of $\phi\otimes\id_A$ to a map 
$C(G,A)\to C_{b,\alpha}([0,\I),A)$. In particular, the 
$G$-action does not play a role anymore.

Using commutativity and separability of $C(G)$,
let $(k_n)_{n\in\N}$ be a sequence in $\N$, and let
$(\rho_n)_{n\in\N}$ and $(\sigma_n)_{n\in\N}$ be 
sequences of unital completely positive maps
$\rho_n\colon C(G)\to \C^{k_n}$ and 
$\sigma_n\colon \C^{k_n}\to A_{\mathrm{c},\alpha}\cap A'$ 
such that 
$\lim\limits_{n\to\I}\|(\sigma_n\circ\rho_n)(f)-\phi(f)\|=0$
for all $f\in C(G)$.
It follows that $(\sigma_n\circ\rho_n)\otimes\id_A$
converges pointwise in norm to $\phi\otimes\id_A$.

By Theorem~6 in~\cite{Arv_notes_1977} and since $C(G,A)$ is separable,
it follows that the set of maps $C(G,A)\to A_{\mathrm{c},\alpha}$
which have unital completely positive lifts is closed in the 
point-norm topology. In particular, it suffices to show that
the map 
\[(\sigma_n\circ\rho_n)\otimes\id_A\colon \C^{k_n}\otimes A\to 
 A_{\mathrm{c},\alpha}
\]
has a unital completely positive lift for every $n\in\N$.

Fix $n\in\N$. Let $e_1,\ldots,e_{k_n}$ denote the canonical
minimal projections in $\C^{k_n}$ adding up to the unit.
Elementary functional calculus allows us to find 
positive contractions 
$d_1,\ldots,d_{k_n}\in C_{b,\alpha}([0,\I),A)$ satisfying
$\pi_A(d_j)=\sigma_n(e_j)$ for all $j=1,\ldots,k_n$ and 
$\sum_{j=1}^{k_n}d_j=1$.
Let $\widetilde{\sigma}_n\colon \C^{k_n}\otimes A\to 
C_{b,\alpha}([0,\I),A)$ be the linear map determined by
\[\widetilde{\sigma}_n(e_j\otimes a)=d_j^{1/2}ad_j^{1/2}\]
for all $j=1,\ldots,k_n$ and all $a\in A$. It is clear that
$\widetilde{\sigma}_n$ is unital and completely positive, 
and that $\pi_A\circ\widetilde{\sigma}_n=\sigma_n\otimes\id_A$. 
Thus
$\widetilde{\sigma}_n\circ (\rho_n\otimes\id_A)$
is a unital completely positive lift of
$(\sigma_n\circ\rho_n)\otimes\id_A$, as desired. This proves that (1) implies (2).

For the last assertion in the theorem, 
we denote by $E^\alpha \colon A\to A^\alpha$ and by
$E^\gamma\colon C(G,A)\to C(G,A)^\gamma$ 
the canonical conditional expectations.
Let $j\colon A\to C(G,A)$ be the homomorphism given by
$j(a)(g)=\alpha_g(a)$ for all $g\in G$ and all $a\in A$, and 
note that $j(A)=C(G,A)^\gamma$. 
For $t\in [0,\I)$, let $\Psi_t\colon A\to A^\alpha$ be the unital, 
completely positive and equivariant map given
by $\Psi_t(a)=E^\alpha(\Upsilon_t(j(a)))$ for all $a\in A$. 
Since $\Upsilon$ is asymptotically equivariant, we get
\[\lim_{t\to\I}\|\Psi_t(a)-\Upsilon_t(E^\gamma(j(a))\|=0\] 
for all $a\in A$. Since $j(a)$ is $\gamma$-invariant, it follows that
$\lim_{t\to\I}\|\Psi_t(a)-\Upsilon_t(j(a))\|=0$ for all $a\in A$. 
Using that the restriction of 
$j$ to $A^\alpha$ is the canonical inclusion $A^\alpha\to C(G,A)$
as constant functions, we get 
\[\lim_{t\to\I}\Psi_t(a)=\lim_{t\to\I}\Upsilon_t(j(a))=a.\]
for all $a\in A^\alpha$. Finally, for $a,b\in A$, we have 
\[\lim_{t\to\I}\|\Psi_t(a)\Psi_t(b)-\Psi_t(ab)\|=
 \lim_{t\to\I}\|\Upsilon_t(j(a))\Upsilon_t(j(b))-\Upsilon_t(j(ab))\|=0,
\]
since $\Upsilon$ is asymptotically multiplicative. This finishes the proof.
\end{proof}

\begin{cor}\label{cor: E-thy summand} 
Let $\alpha\colon G\to\Aut(A)$ be an action of a
second countable compact group $G$ 
with the continuous Rokhlin property on a unital, separable
\ca\ $A$. Let $B$ be any separable \ca, and denote by 
\[\iota^*\colon KK(A,B)\to KK(A^\alpha,B) \ \ \mbox{ and } \ \  
\iota_*\colon KK(B,A^\alpha)\to KK(B,A) 
\]
the group homomorphisms induced by the canonical inclusion $\iota\colon A^\alpha\to A$. Then there exist $\Psi^*\colon KK(A^\alpha,B)\to KK(A,B)$ 
and $\Psi_*\colon KK(B,A)\to KK(B,A^\alpha)$ such that
$\iota^*\circ \Psi^*=\id_{KK(A^\alpha,B)}$ and $\Psi_\ast \circ \iota_*=\id_{KK(B,A^\alpha)}$. In particular,
\[KK(A,B)\cong KK(A^\alpha,B)\oplus \ker(\iota^*) \ \ \mbox{ and } \ \ 
KK(B,A)\cong KK(B,A^\alpha)\oplus \ker(\Psi_*).\]
\end{cor}
\begin{proof} Recall (see Theorem~4.2 in \cite{HouTho_universal_1999}) that given separable \ca s $A$ and $B$, the $KK$-group $KK(A,B)$ is canonically
isomorphic to the group of homotopy
classes of completely positive asymptotic morphisms 
$SA\rightrightarrows SB\otimes\K$.
The unital completely positive asymptotic morphism 
$\Psi\colon A\rightrightarrows A^\alpha$ 
constructed in \autoref{thm: existence of a.m.}
induces a group homomorphism $\Psi^*\colon KK(A^\alpha,B)\to KK(A,B)$ which satisfies
$$\iota^*\circ \Psi^*=\id_{KK(A^\alpha,B)},$$
since $\Psi\circ\iota$ is in fact asymptotically equal to the identity on $A^\alpha$ (not just homotopic). This proves the first claim. The existence of an isomorphism
$KK(A,B)\cong KK(A^\alpha,B)\oplus \ker(\iota^*)$ is then a standard fact in the theory of abelian groups.

The proof for $KK(B,A)$ is analogous, and is left to the reader. \end{proof}

We now turn to preservation of the UCT. 

\begin{df}\label{df: UCT}
Let $A$ be a separable \ca. We say that $A$ \emph{satisfies the UCT} if
for every separable \ca\ $B$, the following conditions are satisfied:
\be\item The natural map $\tau_{A,B}\colon KK(A,B)\to \Hom(K_\ast(A),K_\ast(B))$ defined by Kasparov
in \cite{Kas_equivariant_1988}, is surjective.
\item The natural map $\mu_{A,B}\colon \ker(\tau_{A,B})\to \Ext(K_\ast(A),K_{\ast+1}(B))$ is an isomorphism.\ee
If this is the case, by setting $\ep_{A,B}=\mu_{A,B}^{-1}\colon \Ext(K_\ast(A),K_{\ast+1}(B)) \to KK(A,B)$, we obtain a short exact sequence
$$\xymatrix@=1.35em{0\ar[r] &
\mathrm{Ext}(K_\ast(A),K_{\ast+1}(B))\ar[r]^-{\varepsilon_{A,B}}&
KK(A,B)\ar[r]^-{\tau_{A,B}}&\mathrm{Hom}(K_\ast(A),K_\ast(B))\ar[r]&0,}$$
which is natural on both variables because so are $\tau_{A,B}$ and $\mu_{A,B}$.
\end{df}

The following result generalizes Proposition~3.8 of
\cite{PhiVio_simple_2013} from finite groups to 
compact groups.
We stress the fact that there are no nuclearity 
assumptions in this theorem. 
 
\begin{thm}\label{thm: UCTcRp} 
Let $\alpha\colon G\to\Aut(A)$ be an action of a
second countable compact group $G$ 
with the continuous Rokhlin property on a unital, separable
\ca\ $A$. If $A$ satisfies the UCT, then 
so do $A\rtimes_\alpha G$ and $A^\alpha$.
%Then the following are equivalent:
%\be\item $A$ satisfies the UCT.
%\item The crossed product $A\rtimes_\alpha G$ satisfies the UCT.
%\item The fixed point algebra $A^\alpha$ satisfies the UCT.\ee
\end{thm}
\begin{proof} 
Note that $A\rtimes_\alpha G$ satisfies the UCT if and only if 
so does $A^\alpha$,
by~Corollary~3.12 in~\cite{Gar_classificationI_2014}. 
Assume that $A$ satisfies the UCT; we will show that so does
$A^\alpha$.

Let $B$ be a separable \ca. 
%Given a separable \ca\ $B$,
%there is a short exact sequence
%$$\xymatrix@=1.35em{0\ar[r] &
%\mathrm{Ext}(K_\ast(A),K_{\ast+1}(B))\ar[r]^-{\varepsilon_{A,B}}&
%KK(A,B)\ar[r]^-{\tau_{A,B}}&\mathrm{Hom}(K_\ast(A),K_\ast(B))\ar[r]&0,}$$
%which is natural on both variables. Denote by $\iota\colon A^\alpha\to A$ the canonical inclusion. 
Let $\Psi\colon A\rightrightarrows A^\alpha$ be a unital completely positive
asymptotic morphism as in the conclusion of \autoref{thm: existence of a.m.}. Then $\Psi$ induces group homomorphisms
\begin{align*} \Ext(K_\ast(A^\alpha),K_{\ast+1}(B)) &\to \Ext(K_\ast(A),K_{\ast+1}(B))\\
KK(A^\alpha,B)&\to KK(A,B)\\
\Hom(K_\ast(A^\alpha),K_\ast(B))&\to \Hom(K_\ast(A),K_\ast(B)),\end{align*}
which we will all denote by $\Psi^*$, that are right inverses of the canonical homomorphisms induced by $\iota\colon A^\alpha \to A$ (which we will all denote by $\iota^*$).
The diagrams
\begin{align*} \xymatrix{ \Ext(K_\ast(A),K_{\ast+1}(B)) \ar@/^1pc/[d]^-{\iota^*}&& \ker(\tau_{A,B})\ar[ll]_-{\mu_{A,B}}\ar@/^1pc/[d]^-{\iota^*}\\
 \Ext(K_\ast(A^\alpha),K_{\ast+1}(B))\ar@/^1pc/[u]^-{\Psi^*}& &\ker(\tau_{A^\alpha,B})\ar@/^1pc/[u]^-{\Psi^*}\ar[ll]^-{\mu_{A^\alpha,B}}}\end{align*}
and

\begin{align*} \xymatrix{ KK(A,B) \ar[rr]^-{\tau_{A,B}}\ar@/^1pc/[d]^-{\iota^*}&& \Hom(K_\ast(A),K_\ast(B))\ar@/^1pc/[d]^-{\iota^*}\\
KK(A^\alpha,B) \ar[rr]^-{\tau_{A^\alpha,B}}\ar@/^1pc/[u]^-{\Psi^*}& &\Hom(K_\ast(A^\alpha),K_\ast(B))\ar@/^1pc/[u]^-{\Psi^*}}\end{align*}
are easily seen to be commutative, using naturality of the horizontal maps involved. 

We claim that $\mu_{A^\alpha,B}$ is an isomorphism. Since
$\Psi^*\circ \mu_{A^\alpha,B}=\mu_{A,B}\circ\Psi^*$
and $\Psi^*,\mu_{A,B}$ and $\Psi^*$ are injective, it follows that $\mu_{A^\alpha,B}$ is injective. Surjectivity follows similarly from the identity
$ \mu_{A^\alpha,B}\circ\iota^*=\iota^*\circ \mu_{A,B}$
and the fact that $\iota^*,\mu_{A,B}$ and $\iota^*$ are surjective. The claim is proved.\\
\indent We now claim that $\tau_{A^\alpha,B}$ is surjective. Given $x\in \Hom(K_\ast(A^\alpha),K_\ast(B))$, use surjectivity of $\tau_{A,B}$
to choose $y\in KK(A,B)$ such that $\tau_{A,B}(y)=\Psi^*(x)$. Then
$$(\tau_{A^\alpha,B}\circ\iota^*)(y)=(\iota^*\circ \tau_{A,B})(y)=x,$$
showing that $\tau_{A^\alpha,B}$ is surjective. This proves the claim, and also the theorem.\end{proof}

\begin{rem} Adopt the notation of the theorem above. It is clear that the same argument, verbatim, shows that if
$A$ satisfies the $E$-theoretic version of the UCT, then so do $A^\alpha$ and $A\rtimes_\alpha G$.\end{rem}

It is unclear whether the converse to \autoref{thm: UCTcRp} holds in
general, namely if the UCT for $A^\alpha$ implies the UCT for $A$. 
For \emph{circle} actions, this is always the case:

\begin{cor}
Let $\alpha\colon \T\to\Aut(A)$ be an action with the 
continuous Rokhlin property on a unital, separable \ca\ $A$.
Then the following are equivalent:
\be\item $A$ satisfies the UCT.
\item The crossed product $A\rtimes_\alpha \T$ satisfies the UCT.
\item The fixed point algebra $A^\alpha$ satisfies the UCT.\ee
\end{cor}
\begin{proof} 
It suffices to show that $A$ satisfies the UCT whenever $A^\alpha$ does.
Since $\alpha$ has the Rokhlin property, by Theorem~3.11 in~\cite{Gar_classificationI_2014} there exists an automorphism
$\check{\alpha}\in\Aut(A^\alpha)$ such that $A\cong A^\alpha\rtimes_{\check{\alpha}}\Z$. Thus $A$ satisfies the UCT if $A^\alpha$
does.
\end{proof}

\section{Equivariant \texorpdfstring{$KK$}{KK}-theory for circle actions}
In this section, we specialize to circle actions, since much more can 
be said in this setting. For example, circle actions with the \cRp\ are 
always dual actions, and it can be completely characterized what 
automorphisms arise as (pre)duals of continuous Rokhlin actions; 
see \autoref{thm:duality}. Such automorphisms are in particular 
always trivial in $KK$-theory. In particular, we show that 
if $\alpha\colon\T\to\Aut(A)$ has the \cRp, then $(A,\alpha)$ is
$KK^\T$-equivalent to 
$(C(\T)\otimes A^\alpha, \texttt{Lt}\otimes\id_{A^\alpha})$. As a 
consequence, there are isomorphisms $K_0(A)\cong K_1(A)\cong K_0(A^\alpha)\oplus K_1(A^\alpha)$. 

We begin by restating \autoref{df: cRpG} in terms of unitaries in the 
algebra.

\begin{rem}\label{def cRpT} Let $A$ be a \uca, and let $\alpha\colon\T\to\Aut(A)$ be an action. Then $\alpha$ has the \emph{continuous \Rp} if
and only if 
there exists a continuous path $(u_t)_{t\in[0,\I)}$ of unitaries in $A$ such that
\be\item $\lim\limits_{t\to\I}\sup\limits_{\zeta\in\T}\|\alpha_\zt(u_t)-\zt u_t\|= 0$,
\item $\lim\limits_{t\to\I}\|u_ta-au_t\|= 0$ for all $a\in A$.\ee
\end{rem}

%\begin{rem} Condition (2) in \autoref{def cRp} is satisfied for all $a\in A$ if and only if it is satisfied
%for all elements of some generating set.
%\end{rem}

The next result strengthens condition~(1) in~\autoref{def cRpT}, and will make a number of arguments technically easier. 

\begin{prop}\label{can replace estimate by equality in def of cRp}
Let $A$ be a \uca\ and let $\alpha\colon \T\to\Aut(A)$ be an action. Then
$\alpha$ has the continuous \Rp\ \ifo there exists a continuous path $(u_t)_{t\in [0,\I)}$ of unitaries in $A$ such that
\be\item $\alpha_\zeta(u_t)=\zeta u_t$ for all $\zeta\in\T$ and all $t\in[0,\I)$, and
\item $\lim\limits_{t\to\I}\|u_ta-au_t\|=0$ for all $a\in A$.\ee\end{prop}
\begin{proof}
Choose a path $(v_t)_{t\in [0,\I)}$
of unitaries in $A$ as in \autoref{def cRpT}. Without loss of generality, we may assume that $\|\alpha_\zt(v_t)-\zt v_t\|<\frac{1}{4}$
for all $\zt\in \T$ and all $t\in [0,\I)$. Denote by $\mu$ the normalized Haar measure on $\T$, and for $t\in [0,\I)$, set
$$x_t =\int_{\T}\overline{\zeta}\alpha_\zeta(v_t)\ d\mu(\zeta).$$
Given $t\in [0,\I)$, one checks that $\|x_t\|\leq 1$ and $\|x_t-v_t\|\leq \frac{1}{3}$. Thus $\|x_t^*x_t-1\|< 1$, so $x_t^*x_t$ is
invertible. Set $u_t=x_t(x_t^*x_t)^{-\frac{1}{2}}$, which is a unitary in $A$.\\
\indent For $\zt\in \T$ and $t\in [0,\I)$, it is immediate to check that $\alpha_\zt(x_t)=\zt x_t$, and thus $\alpha_\zt(u_t)=\zt u_t$.
An application of the triangle inequality shows that $\lim\limits_{t\to\I}\|u_ta-au_t\|=0$ for all $a\in A$. Finally,
$$\|x_t-x_s\|= \left\|\int_{\T}\overline{\zeta}\alpha_\zeta(v_t-v_s)\ d\mu(\zeta)\right\|\leq \|v_t-v_s\|$$
for all $t$ and $s\in [0,\I)$, which shows that the map $t\mapsto x_t$ is continuous. This proves that $t\mapsto u_t$ is also continuous,
and hence $(u_t)_{t\in [0,\I)}$ is the desired path. \end{proof}

Let $\alpha\colon\T\to\Aut(A)$ be an action with the \cRp. 
Since $\alpha$ in particular has the \Rp, it 
follows from Theorem~3.11 in~\cite{Gar_classificationI_2014} that $\alpha$
is a dual action, that is, there 
exists a \emph{predual automorphism} $\check{\alpha}\in \Aut(A^\alpha)$
such that $(A,\alpha)$ is equivariantly isomorphic to 
$(A^\alpha\rtimes_{\check{\alpha}}\Z, \widehat{\check{\alpha}})$.
In \autoref{thm:duality}, we completely characterize those automorphisms
that arise as preduals of continuous Rokhlin actions, using the following notion.

%$$\varphi\colon   \to A$$
%such that $\alpha_\zeta=\varphi\circ\widehat{\check{\alpha}}_\zeta\circ\varphi^{-1}$ for all $\zeta\in\T$.

\begin{df}\label{def car} Let $B$ be a \ca\ and let $\beta$ be an automorphism of $B$.
We say that $\beta$ is \emph{asymptotically
representable} if there exists a continuous path $(v_t)_{t\in[0,\I)}$ of contractions in $B$ satisfying the following
conditions
\be
\item $\lim\limits_{t\to\I}\|v_t^*v_tb-b\|=0$ for all $b\in B$;
\item $\lim\limits_{t\to\I}\|v_t^*v_t-v_tv_t^*\| = 0$;
\item $\lim\limits_{t\to\I}\|\beta(v_t)-v_t\|=0$; and
\item $\lim\limits_{t\to\I}\|\beta(b)-v_tbv_t^*\| = 0$ for all $b\in B$.
\ee
\end{df}

\begin{rem} When $B$ is unital and $\beta\in \Aut(B)$ is asymptotically representable, one can use functional calculus to show
that the asymptotically normal contractions $v_t$ in the definition above can be chosen to be \emph{unitaries}.
\end{rem}

%\begin{rem} Let $B$ be a \ca\ and let $\beta$ be an automorphism of $B$. One can easily show that $\beta$ is asymptotically
%representable if and only if there exists a unitary $v$ in
%$$C_b([0,\I),M(B)^\beta)/C_0([0,\I),M(B)^\beta)$$
%such that $\beta(b)=vbv^*$ for all $b\in B$. We leave the proof as an exercise for the reader. We point out that one does not need to
%assume the \ca\ $B$ to be separable, unlike in the case of approximately representable automorphisms. \end{rem}

We fix some notation that will be used later.

\begin{rem}\label{rem:ConvProduct}
Endow $\T$ with its normalized Haar measure.
We endow $L^1(\T)$ with the convolution product, and regard it
as a dense subalgebra of $C^*(\T)$. More generally,
if $\alpha\colon\T\to\Aut(A)$ is a circle action on a \uca\ $A$,
we endow $L^1(\T,A)$ with the operations
\[(\xi\ast\eta)(\zeta)=\int_\T \xi(\omega)\alpha_\omega(\eta(\omega^{-1}\zeta))\ d\omega \ \ \mbox{ and } \ \ \xi^*(\zeta)=\alpha_\zeta(\xi(\overline{\zeta})^*)\]
for all $\xi,\eta\in L^1(\T,A)$ and all $\zeta\in\T$. 
Then $L^1(\T,A)$ is naturally a dense
$\ast$-subalgebra of $A\rtimes_\alpha\T$, and the 
$L^1$-norm on $L^1(\T,A)$ dominates the $C^*$-norm. 
There is a canonical inclusion
$C^*(\T)\subseteq A\rtimes_\alpha \T$, 
and any approximate identity for $C^*(\T)$ is also an approximate
identity for $A\rtimes_\alpha\T$.

Given $f\in L^1(\T)$ and 
$a\in A$ we write $fa\in L^1(\T,A)$ for the function 
given by $(fa)(\zeta)=f(\zeta)a$ for all $\zeta\in\T$.
The linear span of the 
elements of this form is dense in $L^1(\T,A)$ and
hence also in $A\rtimes_\alpha\T$. Finally, the dual automorphism
$\widehat{\alpha}\in\Aut(A\rtimes_\alpha\T)$ is determined by 
$\widehat{\alpha}(fa)(\zeta)=\zeta f(\zeta)a$ for all 
$\zeta\in\T$. 
\end{rem}

We proceed to show that asymptotic representability is the notion dual to the continuous \Rp.

%, in complete analogy with the duality between
%the \Rp\ and approximate representability. 
%Recall that if $A$ is a unital \ca\ and $G$ is a locally compact group acting on $A$, then
%the crossed product $A\rtimes G$ is unital if and only if $G$ is discrete. When $A$ is unital, there is always a canonical inclusion
%$C^*(G)\subseteq A\rtimes G$, and any approximate unit in $C^*(G)$ is also an approximate unit in $A\rtimes G$.

\begin{thm}\label{thm:duality}
Let $A$ be a \uca, let $\alpha\colon\T\to\Aut(A)$ be an action, and let $\beta\in \Aut(A)$ be an automorphism.
\be
\item The action $\alpha$ has the \cRp\ if and only if $\widehat{\alpha}\in\Aut(A\rtimes_\alpha \T)$ is asymptotically representable.
\item The automorphism $\beta$ is asymptotically representable if and only if $\widehat{\beta}\colon \T\to\Aut(A\rtimes_\beta\Z)$ has the
continuous Rokhlin property.
\ee
\end{thm}
\begin{proof}
(1) Assume that $\alpha$ has the \cRp, and let $(u_t)_{t\in [0,\I)}$ be a continuous path of unitaries as in
\autoref{can replace estimate by equality in def of cRp}. 
Let $\mu$ denote the normalized Haar measure on $\T$.
For $t\in [0,\I)$, let $K_t\subseteq \T$ denote the compact
symmetric 
neighborhood of the identity with $\mu(K_t)=1/(t+1)$, and 
let $f_t\colon \T\to \C$ be the 
positive function with $\supp(f_t)=K_t$ and $f_t(1)=2(t+1)$, 
and is otherwise linear. Since $f_t$ is integrable, we 
regard it as an element in 
$C^*(\T)\subseteq A\rtimes_\alpha\T$. We list some properties
of $(f_t)_{t\in [0,\I)}$ that will be used in the sequel:
\be\item[(a)] $(f_t\ast f_t)_{t\in [0,\I)}$ is an approximate unit for $A\rtimes_\alpha\T$ (because $(f_t)_{t\in [0,\I)}$ is);
\item[(b)] $\|f_t\|_1=1=\|f_t\ast f_t\|_1$ for all $t\in [0,\I)$;
\item[(c)] $f_t$ is symmetric: $f_t(\zeta)=\overline{f_t(\overline{\zeta})}$ for all 
$\zeta\in\T$;
\item[(d)] With $\widetilde{f}_t(\zeta)=\zeta f_t(\zeta)$ for all
$\zeta\in\T$, we have 
\[\lim_{t\to\I}\|f_t-\widetilde{f}_t\|_1=0=
 \lim_{t\to\I}\|f_t\ast f_t-\widetilde{f}_t\ast \widetilde{f}_t\|_1;
\]
\item[(e)] Given $f\in C(\T)$, given $a\in A$, and $\ep>0$,
there exists $t_0\in [0,\I)$ such that 
\[\|\alpha_\omega(a)-a\|<\ep \ \ \mbox{ 
 and } \ \ |f(\sigma)-f(\zeta)|<\ep
\]
whenever $t\geq t_0$ and 
$\omega, \overline{\zeta}\sigma\omega\in \supp(f_t)=K_t$.
\ee 

For $t\in [0,\I)$, set $v_t=f_tu_t^*$, which in view
of the last comment in \autoref{rem:ConvProduct}
is a contraction in $A\rtimes_\alpha\T$. Moreover, the map $t\mapsto v_t$
is continuous. Let $\zeta\in\T$. 
Using that $\alpha_\zeta(u_t^*)=\overline{\zeta}u_t^*$
at the fourth step, we get
\[v_t^*(\zeta)=\alpha_\zeta(v_t(\overline{\zeta})^*)=\alpha_\zeta(\overline{f_t(\overline{\zeta})}u_t)
\stackrel{(c)}{=}
 f_t(\zeta)\alpha_\zeta(u_t)=
 \zeta f_t(\zeta)u_t.
\]

We proceed to check the conditions in \autoref{def car}. 
For $\zeta\in\T$, we have
\begin{align*}
(v_t^*\ast v_t)(\zeta)
&= \int_\T v_t^*(\omega)\alpha_\omega(v_t(\overline{\omega}\zeta))\ d\omega\\
&= \int_\T \omega f_t(\omega)u_t\alpha_\omega(f_t(\overline{\omega}\zeta)u_t^*)\ d\omega\\
&= \int_\T \omega f_t(\omega)u_tf_t(\overline{\omega}\zeta)\overline{\omega}u_t^*\ d\omega\\
&= (f_t\ast f_t)(\zeta).
\end{align*}
Thus $v_t^*\ast v_t=f_t\ast f_t$ for all $t\in [0,\I)$. 
Since $(f_t\ast f_t)_{t\in [0,\I)}$ is an approximate identity
for $A\rtimes_\alpha\T$ by (a), condition (1) is satisfied.
We turn to (2). For $\zeta\in\T$ and $t\in [0,\I)$, we have
\begin{align*}
(v_t\ast v_t^*)(\zeta)
&= \int_\T v_t(\omega)\alpha_\omega(v_t^*(\overline{\omega}\zeta))\ d\omega\\
&= \int_\T f_t(\omega)u_t^*\alpha_\omega(\overline{\omega}\zeta f_t(\overline{\omega}\zeta)u_t)\ d\omega\\
&= \int_\T f_t(\omega)u_t^*\overline{\omega}\zeta f_t(\overline{\omega}\zeta)\omega u_t\ d\omega\\
&= \int_\T \omega f_t(\omega)\overline{\omega}\zeta f_t(\overline{\omega}\zeta)\ 
d\omega\\
&= (\widetilde{f}_t\ast \widetilde{f}_t)(\zeta).
\end{align*}
Thus $v_t\ast v_t^*=\widetilde{f}_t\ast \widetilde{f}_t$. 
Since
$\|\cdot\|\leq \|\cdot\|_1$, it follows from (d) that 
$\widetilde{f}_t\ast \widetilde{f}_t$ is 
asymptotically equal to $f_t\ast f_t=v_t^*\ast v_t$ (in
the norm of $A\rtimes_\alpha\T$). 
Thus condition~(2) in \autoref{def car} is satisfied.
In order to check condition~(3), 
let $\zeta\in\T$. Then
\[\widehat{\alpha}(v_t)(\zeta)=\zeta f_t(\zeta) u_t^*=\widetilde{f}_t(\zeta)u_t^*.\]
Using the identity above at the second step, it follows that 
\[
0\stackrel{(d)}{=}\lim_{t\to\I}\|\widetilde{f}_t-f_t\|_1=\lim_{t\to\I}\|\widehat{\alpha}(v_t)-v_t\|_1\geq \lim_{t\to\I}\|\widehat{\alpha}(v_t)-v_t\|,
\]
thus verifying~(3). 
Finally, to check~(4), it suffices to 
take $b=fa$ for $f\in C(\T)\subseteq L^1(\T)$ and $a\in A$.
Let $\ep>0$. Using (e), find $t_0$ large enough so that
\[\|\alpha_\omega(a)-a\|<\frac{\ep}{2} \ \ \mbox{ 
 and } \ \ |f(\sigma)-f(\zeta)|<\frac{\ep}{2}
\]
whenever $t\geq t_0$ and 
$\omega, \overline{\zeta}\sigma\omega\in \supp(f_t)$.
Given
$\zeta\in\T$, we have
\begin{align*}(v_t\ast b \ast v_t^*)(\zeta)
&=\int_\T v_t(\omega)\alpha_\omega((b \ast v_t^*)(\overline{\omega}\zeta)) \ d\omega \\
&=\int_\T f_t(\omega)u_t^*\alpha_\omega\left( 
\int_\T 
f(\sigma)a \alpha_\sigma(v_t^*(\overline{\sigma\omega}\zeta)) 
\ d\sigma\right) d\omega \\
&=\int_\T\int_\T f_t(\omega)u_t^*f(\sigma)\alpha_\omega(a)
\alpha_{\omega\sigma 
}\left(\overline{\sigma\omega}\zeta
f_t(\overline{\sigma\omega}\zeta)u_t
\right) d\sigma d\omega \\
&=\int_\T\int_\T f_t(\omega)u_t^* 
f(\sigma)\alpha_\omega(a) 
\overline{\sigma\omega}\zeta
f_t(\overline{\sigma\omega}\zeta)
\omega\sigma u_t
\ d\sigma d\omega \\
&=\zeta\int_\T\int_\T f_t(\omega) 
f(\sigma)\alpha_\omega(a) 
f_t(\overline{\sigma\omega}\zeta)
\ d\sigma d\omega. 
\end{align*}
Combining the previous computation with the choice of $t_0$, 
we get the following
approximation in the norm of $A$ whenever $t\geq t_0$:
\begin{align*}
(v_t\ast b \ast v_t^*)(\zeta)&\approx_{\ep} 
\zeta f(\zeta)a\int_\T\int_\T f_t(\omega) 
f_t(\overline{\sigma\omega}\zeta)
\ d\sigma d\omega \\
&= \zeta b(\zeta)
\int_\T (f_t\ast f_t)(\overline{\sigma}\zeta)\ d\sigma\\
&= \zeta b(\zeta)\|f_t\ast f_t\|_1\stackrel{(b)}{=}
\widehat{\alpha}(b)(\zeta).\end{align*}
Since the bound is uniform on $\zeta\in\T$, for all $t\geq t_0$ it follows that 
\[\ep\geq \|v_t\ast b \ast v_t^*-\widehat{\alpha}(b)\|_1
 \geq \|v_t\ast b \ast v_t^*-\widehat{\alpha}(b)\|.
\]
Since $\ep>0$ is arbitrary, this verifies condition~(4),
and shows that $\widehat{\alpha}$ is asymptotically representable.

Conversely, assume that $\widehat{\alpha}$ is asymptotically representable, and let $(v_t)_{t\in [0,\I)}$ be
a continuous path in $A\rtimes_\alpha\T$ satisfying the conditions in \autoref{def car}.
Let $w\in M(A\rtimes_\alpha\T\rtimes_{\widehat{\alpha}}\Z)$ be the canonical unitary implementing $\widehat{\alpha}$,
and set $\widetilde{u}_t=v_t^*w$ for all $t\in [0,\I)$, which we regard as an element in $A\rtimes_\alpha\T\rtimes_{\widehat{\alpha}}\Z$.
Given $x\in A\rtimes_\alpha\T$ and $k\in\Z\setminus\{0\}$, we use  that $v_t$ asymptotically commutes with $w$ at
the second step (since $\lim\limits_{t\to\I}\|\widehat{\alpha}(v_t)-v_t\|=0$); that $w$ implements $\widehat{\alpha}$ at the second step; 
and the fact that $\widehat{\alpha}^{-1}$ is asymptotically implemented by $v_t^*$ at the third step, to get:
\begin{align*}\lim_{t\to\I} \widetilde{u}_t(xw^k)-(xw^k)\widetilde{u}_t&=\lim_{t\to\I} v_t^*wxw^k-xw^kv_t^*w\\
&=\lim_{t\to\I} v_t^*\widehat{\alpha}(x)w^{k+1}-xv_t^*w^{k+1}\\
&=\lim_{t\to\I} xv_t^*w^{k+1}-xv_t^*w^{k+1}=0. \end{align*}

Since the set $\{xw^k\colon x\in A\rtimes_\alpha\T, k\in\Z\setminus\{0\}\}$ generates $A\rtimes_\alpha\T\rtimes_{\widehat{\alpha}}\Z$
as a \ca, we conclude that $\widetilde{u}_t$ is asymptotically central. On the other hand, it is clear that
$\widehat{\widehat{\alpha}}_\zeta(\widetilde{u}_t)=\zeta\widetilde{u}_t$ for all $\zeta\in\T$ and all $t\in [0,\I)$.

Use Takai duality to identify $A\rtimes_\alpha\T\rtimes_{\widehat{\alpha}}\Z$ with $A\otimes\K(L^2(\T))$. Let
$p\in \K(L^2(\T))$ be the projection onto the constant functions, and let $e\in M(A\rtimes_\alpha\T\rtimes_{\widehat{\alpha}}\Z)$
be the projection corresponding to $1_A\otimes p$. Then $\widehat{\widehat{\alpha}}_\zeta(exe)=e\widehat{\widehat{\alpha}}_\zeta(x)e$
for all $x\in A\rtimes_\alpha\T\rtimes_{\widehat{\alpha}}\Z$, and there is a canonical equivariant isomorphism
\[\left(e(A\rtimes_\alpha\T\rtimes_{\widehat{\alpha}}\Z)e,\widehat{\widehat{\alpha}}\right)\cong (A,\alpha).\]

For $t\in [0,\I)$, set $u_t=e\widetilde{u}_te$, which we regard as an element in $A$. Then
$t\mapsto u_t$ is a continuous path in $A$, and $\alpha_\zeta(u_t)=\zeta u_t$ for all $\zeta\in\T$ and all $t\in [0,\I)$.
Moreover,
\[\lim_{t\to\I}u_tau_t^*=\lim_{t\to\I}e\widetilde{u}_teae\widetilde{u}_t^*e=e(eae)e=a\]
for all $a\in A$. In particular, putting $a=1_A$ we get $\lim\limits_{t\to\I}u_t^*u_t=1_A$, and similarly for $u_tu_t^*$.
Using functional calculus, we can perturb $u_t$ to a nearby continuous path of unitaries in $A$ which satisfy
the conditions in \autoref{def cRpT}. This shows that $\alpha$ has the continuous \Rp.

(2) Assume that $\beta$ is asymptotically representable. Let $(v_t)_{t\in [0,\I)}$ be a continuous path of unitaries in $A$
satisfying
$\lim\limits_{t\to\I}\|\beta(a)-v_tav_t^*\|=0$ for all $a\in A$ and $\lim\limits_{t\to\I}\|\beta(v_t)-v_t\|=0$.
Denote by $w$ the canonical unitary in $A\rtimes_\beta\Z$ that implements $\beta$. For $t\in [0,\I)$, set $u_t=v_t^*w$, which
is a unitary in $A\rtimes_\beta\Z$. Moreover, for $\zt\in \T$ we have $\widehat{\beta}_\zt(u_t)=\zt u_t$, so condition (1)
of \autoref{def cRpT} is satisfied for $\widehat{\beta}$ with $(u_t)_{t\in [0,\I)}$. To check condition (2), it is enough
to consider elements in $A\cup\{w\}$. For $a\in A$, we have
\[\lim\limits_{t\to\I}u_tau_t^*=\lim\limits_{t\to\I}v_t^*waw^*v_t=\lim\limits_{t\to\I}v_t^*\beta(a)v_t = \beta^{-1}(\beta(a))=a,\]
and hence $\lim\limits_{t\to\I}\|u_ta-au_t\|=0$, as desired. Finally,
\[\lim\limits_{t\to\I}\|u_twu_t^*-w\|=\lim\limits_{t\to\I}\|v_t^*wv_t-w\|=\lim\limits_{t\to\I}\|wv_tw^*-v_t\|=\lim\limits_{t\to\I}\|\beta(v_t)-v_t\|=0.\]
We conclude that $\widehat{\beta}$ has the continuous Rokhlin property.

Conversely, assume that $\widehat{\beta}$ has the continuous Rokhlin property. Use \autoref{can replace estimate by equality in def of cRp}
to choose a continuous path $(u_t)_{t\in [0,\I)}$ of unitaries in $A\rtimes_\beta\Z$ such that
\bi\item $\widehat{\beta}_\zt(u_t)=\zeta u_t$ for all $\zt\in \T$ and all $t\in [0,\I)$;
\item $\lim\limits_{t\to\I}\|u_tb-bu_t\|=0$ for all $b\in A\rtimes_\beta\Z$.\ei

We continue to denote by $w$ the canonical unitary in $A\rtimes_\beta\Z$ that implements $\beta$.
For $t\in [0,\I)$, set $v_t=wu_t^*$, which is a unitary in $A\rtimes_\beta\Z$. For $\zt\in \T$,
we have $\widehat{\beta}_\zt(v_t)=v_t$, so $v_t$ belongs to $(A\rtimes_\beta\Z)^{\widehat{\beta}}$, which equals $A$
by Proposition~7.8.9 in~\cite{Ped_algebras_1979}. 
Moreover, for $a\in A$, we have
\[\lim\limits_{t\to\I}\|v_tav_t^*-\beta(a)\|=\lim\limits_{t\to\I}\|wu_t^*au_tw^*-waw^*\|=\lim\limits_{t\to\I}\|u_t^*au_t-a\|= 0.\]
Additionally, $\lim\limits_{t\to\I}\|\beta(v_t)-v_t\|=\lim\limits_{t\to\I}\|w(wu_t^*)w^*-wu_t^*\|=0$, because $u_t$ asymptotically commutes with the canonical unitary $w\in A\rtimes_\beta\Z$.

It follows that the continuous path $(v_t)_{t\in [0,\I)}$ of unitaries in $A$ satisfies the conditions of \autoref{def car},
and hence $\beta$ is asymptotically representable.
\end{proof}

\autoref{thm:duality} can be used to produce interesting examples
of actions with the \cRp. For example, if $(u_n)_{n\in\N}$ is any
sequence of unitaries in the CAR algebra $M_{2^\I}$ coming from 
matrix subalgebras, then 
$\varphi=\bigotimes\limits_{n\in\N}\Ad(u_n)$ defines an asymptotically
representable automorphism of $M_{2^\I}$.
Hence its dual
action $\alpha=\widehat{\varphi}\colon\T\to\Aut(M_{2^\I}\rtimes_\varphi\Z)$
has the \cRp. 

Recall that an automorphism $\varphi$ of a unital \ca\ $A$ is 
said to be
\emph{asymptotically inner} if there exists a continuous path $(u_t)_{t\in [0,\I)}$ of unitaries in $A$ such that $\varphi(a)=\lim\limits_{t\to\I}u_tau_t^*$ for all $a\in A$. Since asymptotically
representable automorphisms are clearly asymptotically inner,
we deduce the following useful consequence of 
\autoref{thm:duality}:

\begin{cor}\label{cor:PredualAsymptInner}
Let $\alpha\colon \T\to\Aut(A)$ be an action 
with the continuous Rokhlin property on a unital
\ca\ $A$. Then $\widehat{\alpha}$ and $\check{\alpha}$
are asymptotically inner, and in particular induce the trivial
element in $KK$-theory.
\end{cor}

\begin{df}\label{df:UniKKEq}
Let $\alpha\colon\T\to\Aut(A)$ and $\beta\colon \T\to\Aut(B)$
be circle actions on unital \ca s.
We say that $(A,\alpha)$ and $(B,\beta)$ are \emph{unitally $KK^\T$-equivalent},
if there exists an invertible element $\eta\in KK^\T(A,B)$
with $[1_A]\times \eta=[1_B]\in K_0^\T(B)$
\footnote{In the expression $[1_A]\times \eta$ we use the 
Kasparov product
$KK^\T(\C,A)\times KK^\T(A,B)\to KK^\T(\C,B)$, 
once we identify $K_0^\T(A)$, where $[1_A]$ naturally belongs,
with $KK^\T(\C,A)$, and similarly for $B$.}. 
In this case, we call $\eta$ a
\emph{unital $KK^\T$-equivalence}.
\end{df}

We close this section by showing that circle actions with the 
\cRp\ are always unitally $KK^\T$-equivalent to a (trivial) 
amplification of $\texttt{Lt}$. 

\begin{prop}\label{prop:KKTequivalenceCTAalpha}
Let $\alpha\colon \T\to\Aut(A)$ be an action 
with the continuous Rokhlin property on a unital, separable
\ca\ $A$. 
Then $(A,\alpha)$ is unitally $KK^\T$-equivalent to $(C(\T)\otimes A^\alpha,
\texttt{Lt}\otimes\id_{A^\alpha})$. In particular,
$K_0(A)\cong K_1(A)$ and there is an isomorphism
$K_0(A)\cong K_0(A^\alpha)\oplus K_1(A^\alpha)$ satisfying
$[1_A]\mapsto ([1_{A^\alpha}],0)$.
\end{prop}
\begin{proof}
Let $\Upsilon\colon C(\T,A)\rightrightarrows A$ be the unital
equivariant asymptotic morphism constructed in \autoref{thm: existence of a.m.}. Denote by $\iota\colon A^\alpha\to A$ the canonical inclusion. 
Set $(B,\beta)=(C(\T, A^\alpha),\texttt{Lt}\otimes\id_{A^\alpha})$ and 
\[\eta=[\Upsilon\circ ( \id_{C(\T)}\otimes \iota)]\in 
KK^\T(B,A).\] 
We claim that 
$\eta$ is a unital $KK^\T$-equivalence. Note that unitality
is guaranteed by the fact that $\iota$ and $\Upsilon$ are
unital, so it remains to prove that $\eta$ is a 
$KK^\T$-equivalence. Let 
$\eta\rtimes\T\in KK^{\Z}(B\rtimes_\beta \T, A\rtimes_\alpha\T)$
denote the class induced by $\eta$; see 
Theorem~3.11 in~\cite{Kas_equivariant_1988}.
By the Baaj-Skandalis duality (see Corollary~6.21 in \cite{BaaSka_algebres_1989}),
the assignment $\eta\mapsto \eta\rtimes\T$ defines a natural isomorphism
$KK^\T(B,A)\cong KK^{\Z}(B\rtimes_\beta \T, A\rtimes_\alpha\T)$,
and thus it suffices to check
that $\eta\rtimes \T$ is a $KK^\Z$-equivalence. 
Moreover, by the second
paragraph on page 287 of~\cite{NesMey_homological_2010}, it
is enough to prove that $\eta\rtimes \T$ is a $KK$-equivalence
between $B\rtimes_\beta\T=C(\T,A^\alpha)\rtimes \T\cong A^\alpha\otimes\K(L^2(\T))$ and $A\rtimes_\alpha\T$.

Write $e\in L^1(\T)\subseteq C^*(\T)$ for the constant
function with value 1, and
and write $\delta_0\in c_0(\Z)$ for the Dirac function at
$0\in\Z$. 
Note that the Gelfand transform $C^*(\T)\cong c_0(\Z)$
maps $e$ to $\delta_0$.
Identifying $(A,\alpha)$ with $(A^\alpha\rtimes_{\check{\alpha}}\Z,\widehat{\check{\alpha}})$, 
denote by $\mathcal{T}\colon A\rtimes_\alpha\T\to A^\alpha\otimes\K(\ell^2(\Z))$ the isomorphism given by Takai duality.
Since 
$\Upsilon(1_{C(\T)}\otimes \iota(a))=\iota(a)$ for all $a\in A^\alpha$, the
following diagram is commutative:
\begin{align*}
\xymatrix{
&& A^\alpha\otimes \K(L^2(\T))\ar@<-.5ex>[d] \ar@<.5ex>[d]^{\eta\rtimes\T}\\
A^\alpha\ar[urr]^-{\id_{A^\alpha}\otimes e \ } 
\ar[drr]_-{\id_{A^\alpha}\otimes \delta_0 \ \ }
&& A\rtimes_\alpha\T \ar[d]^-{\mathcal{T}}\\
&& A^\alpha\otimes\K(\ell^2(\Z)).
}
\end{align*}
Since $\id_{A^\alpha}\otimes e$, $\id_{A^\alpha}\otimes \delta_0$ and $\mathcal{T}$ all induce $KK$-equivalences,
it follows that $[\eta\rtimes\T]$ is also a $KK$-equivalence.
By the preceeding comments, it follows that $\eta$ is a 
$KK^\T$-equivalence, finishing the proof.
\end{proof}

We deduce that unital $KK^\T$-equivalence for $\T$-actions with the \cRp\ is equivalent to unital $KK$-equivalence
of the fixed point algebras:

\begin{cor}\label{cor:KKTequivKKeq}
Let $\alpha\colon\T\to\Aut(A)$ and $\beta\colon\T\to\Aut(B)$
be actions with the \cRp\ on unital, separable \ca s $A$ and 
$B$. Then $(A,\alpha)\sim_{KK^{\T}}(B,\beta)$ unitally
if and only if $A^\alpha\sim_{KK}B^\beta$ unitally. 
\end{cor}
\begin{proof} 
Assume that $(A,\alpha)\sim_{KK^{\T}}(B,\beta)$ unitally.
By taking crossed products, there is a $KK$-equivalence
$A\rtimes_\alpha \T\sim_{KK} B\rtimes_{\beta}\T$.
Since $\alpha$ is itself a dual action, it follows 
from Takai duality that
$A\rtimes_\alpha\T\cong A^\alpha\otimes\K(\ell^2(\Z))$
canonically, and similarly for $\beta$. Moreover, 
unitality
of the original $KK^\T$-equivalence implies 
that the induced $KK$-equivalence
\[A^\alpha\otimes\K(\ell^2(\Z))\sim_{KK}
 B^\alpha\otimes\K(\ell^2(\Z))
\]
can be restricted to suitable corners to get a 
unital $KK$-equivalence $A^\alpha\sim_{KK}B^\beta$, as desired.

Conversely, a unital $KK$-equivalence $A^\alpha\sim_{KK}B^\beta$
can be tensored with the identity on $C(\T)$ to obtain a
unital $KK$-equivalence $C(\T,A^\alpha)\sim_{KK}C(\T,B^\beta)$.
This $KK$-class is equivariant with respect to 
$\texttt{Lt}\otimes\id_{A^\alpha}$ and 
$\texttt{Lt}\otimes\id_{B^\beta}$, and thus determines a
unit-preserving $KK^\T$-class
\[\xi\in KK^\T\left((C(\T,A^\alpha),\texttt{Lt}\otimes\id_{A^\alpha}),(C(\T,B^\beta),\texttt{Lt}\otimes\id_{A^\alpha})\right).\]
Since $\xi$ is a $KK$-equivalence between $C(\T,A^\alpha)$
and $C(\T,B^\beta)$, it follows from 
the Baaj-Skandalis' duality and the Pimsner-Voiculescu
exact sequence in $KK$-theory (see specifically the comments 
on page 287 of~\cite{NesMey_homological_2010}) that $\xi$ 
is a $KK^\T$-equivalence.
Since $(C(\T,A^\alpha),\texttt{Lt}\otimes\id_{A^\alpha})$ is 
unitally $KK^\T$-equivalent to $(A,\alpha)$ by \autoref{prop:KKTequivalenceCTAalpha}, and similarly for 
$(B,\beta)$, it follows that $(A,\alpha)\sim_{KK^\T}(B,\beta)$
unitally. 
\end{proof}

We denote by $R(\T)$ the representation ring of $\T$, which is
naturally isomorphic to the ring $\Z[x,x^{-1}]$ of Laurent 
polynomials. Thus, an $R(\T)$-module is simply an abelian group
together with a distinguished automorphism (which is induced
by multiplication by $x$). If $\alpha\colon \T\to\Aut(A)$ is an
action on a \ca\ $A$, then the $\T$-equivariant $K$-theory
$K_\ast^\T(A,\alpha)$ is canonically an $R(\T)$-module: as a 
group, it is isomorphic to $K_\ast(A\rtimes_\alpha \T)$, and
the distinguished automorphism is induced by $\widehat{\alpha}$.

\begin{rem}
When $A$ satisfies the UCT, 
\autoref{prop:KKTequivalenceCTAalpha} can be easily obtained using
the results of Bentmann-Meyer \cite{BenMey_general_2017} on classification up to 
$KK^\T$-equivalence. Indeed, if $A$ satisfies the UCT, then $(A,\alpha)$ belongs to the $\T$-equivariant
bootstrap class thanks to \autoref{thm: UCTcRp} and Proposition~3.1 
in~\cite{BenMey_general_2017}. 
We use Theorem~2.6 in~\cite{BenMey_general_2017} with $\mathfrak{A}$ being the 
category of $\Z_2$-graded $R(\T)$-modules (where $K_\ast^\T(A,\alpha)$ naturally lives).
Set $B=C(\T)\otimes A^\alpha$ and $\beta=\texttt{Lt}\otimes\id_{A^\alpha}$. Then $\beta$ has the \cRp.
Since the equivariant $K$-theories of $\alpha$ and $\beta$ are 
canonically isomorphic to $K_\ast(A^\alpha)$,
it follows from Theorem~2.6 in~\cite{BenMey_general_2017}
that $\alpha\sim_{KK^\T}\beta$ if 
and only if the 
canonical classes that $\alpha$ and $\beta$ determine in 
$\mathrm{Ext}_{R(\T)}(K_\ast(A^\alpha),K_\ast(SA^\alpha))$ agree. 
On the other hand, the 
$R(\T)$-actions are trivial and hence there is an identification
\[\mathrm{Ext}_{R(\T)}(K_\ast(A^\alpha),K_\ast(SA^\alpha))\cong \mathrm{Ext}_{\Z}(K_\ast(A^\alpha),K_\ast(SA^\alpha)).\] 
Moreover, an inspection of the proof of 
Theorem~2.6 in~\cite{BenMey_general_2017} shows that the Ext-class that 
$\alpha$ induces is the one coming from the Pimsner-Voiculescu
exact sequence for its predual automorphism $\check{\alpha}$. 
Since $\check{\alpha}$ is $KK$-trivial by \autoref{thm:duality}, this 
Ext-class is trivial, and thus uniquely determined. We conclude that
$(A,\alpha)\sim_{KK^\T} (B,\beta)$.
\end{rem}

%Recall that a \emph{Kirchberg algebra} is a simple, 
%purely infinite, separable, nuclear \ca. 
%Denote by $\mathcal{S}$
%the (unique) Kirchberg algebra satisfying the UCT 
%which is $KK$-equivalent to $C_0(\mathbb{R})$; see Theorem~4.2.5 %in~\cite{Phi_classification_2000} for existence of
%$\mathcal{S}$, and see
%Theorem~4.2.1
%in~\cite{Phi_classification_2000} its uniqueness.
We isolate the following notion for later use.

\begin{df}\label{df:KKsymmetric}
We say that a \uca\ $A$ is \emph{(nuclearly) 
$KK$-symmetric} if there exist
a unital (nuclear) \ca\ $B$ such that $A\sim_{KK}C(\T,B)$
unitally.
\end{df}

Note that a $KK$-symmetric \ca\ $A$ satisfies
$A\sim_{KK} SA$, and therefore
$K_0(A)\cong K_1(A)$. A \ca\ satisfying the UCT is
nuclearly $KK$-symmetric if and only if 
$K_0(A)\cong K_1(A)$.
%The following 
%is an immediate consequence of \autoref{prop:KKTequivalenceCTAalpha}.

\begin{cor}\label{cor:KKsymmetric}
Let $\alpha\colon \T\to\Aut(A)$ be an action 
with the continuous Rokhlin property on a unital, separable
\ca\ $A$. Then $A$ is $KK$-symmetric. 
If $A$ is nuclear, then it is nuclearly $KK$-symmetric.
\end{cor}
\begin{proof}
This follows from 
\autoref{prop:KKTequivalenceCTAalpha}, since $A^\alpha$ is
nuclear whenever $A$ is. 
\end{proof}

For Kirchberg algebras, we will see in the following section 
that the converse to \autoref{cor:KKsymmetric} is also true.

\section{Classification theorems for circle actions on Kirchberg algebras}
In this section we continue to work with circle actions with the 
\cRp, 
and we prove results related to their classification 
in the case of 
Kirchberg algebras, that is, simple, purely infinite, separable
and nuclear \ca s.
Our results are as follows: a unital Kirchberg algebra 
admits a circle action with the \cRp\ if and only if it is 
nuclearly $KK$-symmetric 
(\autoref{thm: rangeInvCRp});
two circle actions on a unital Kirchberg algebra with the \cRp\ are 
conjugate if and only if they are unitally $KK^\T$-equivalent (\autoref{thm: ClassCRp}); any action with the \cRp\ is unitally $KK^\T$-equivalent to a (necessarily unique)
action on a unital Kirchberg algebra with the \cRp\ (\autoref{thm:EverycRpKKequivKirch}). We also
completely describe the possible values of the $\T$-equivariant $K$-theory
in this setting (\autoref{thm: rangeInvCRp}). 
The results in this section
are $\T$-equivariant versions of celebrated results of 
Kirchberg and Kirchberg-Phillips. 

In the following theorem, the fact that $KK^\T$-equivalence reduces to having isomorphic 
equivariant $K$-theory in the UCT case
is by no means obvious. Indeed, there is no ``equivariant UCT'', 
and there exist examples of circle actions on UCT algebras 
with isomorphic $K^\T$-theory which are not $KK^\T$-equivalent.
%In the following theorem, we use $\T$-equivariant $KK$-theory as the invariant. When the underlying algebra satisfies the UCT, 
%the invariant reduces to $\T$-equivariant $K$-theory.

%We refer the reader
%to \cite{Phi_equivariant_1987} for a thorough development of
%equivariant $K$-theory. 
%In particular, we denote by $R(G)$ the group
%ring associated to the compact group $G$, which is the set of all formal
%finite linear combinations of equivalence classes of irreducible %representations of $G$. For $G=\T$, we have $R(\T)\cong \Z[x,x^{-1}]$,
%where the action of $x$ on 
%$K_\ast^\alpha(A)\cong K_\ast(A\rtimes_\alpha\T)$ is given by $K_\ast(\widehat{\alpha})$.

Recall that two automorphisms $\varphi$ and $\psi$ of a unital \ca\ $A$ are said to be \emph{cocycle conjugate} if there exists $\theta\in\Aut(A)$ such that $\theta\circ\varphi\circ\theta^{-1}$ and $\psi$ are unitarily equivalent. An automorphism $\varphi$ of $A$ is said to be \emph{aperiodic}
if all of its (nonzero) powers are outer. 

\begin{thm}\label{thm: ClassCRp} 
Let $\alpha\colon \T\to\Aut(A)$ and $\beta\colon \T\to\Aut(B)$ be actions
with the continuous \Rp\ on unital Kirchberg 
algebras $A$ and $B$. 
Then the following are equivalent:
\be\item $(A,\alpha)$ and $(B,\beta)$ are conjugate;
\item $(A,\alpha)$ and $(B,\beta)$ are unitally $KK^\T$-equivalent;
\item $A^\alpha$ is isomorphic to $B^\beta$.\ee 
When $A$ and $B$
satisfy the UCT, the above are also equivalent to 
$(K_\ast^\T(A,\alpha),[1_A])\cong (K_\ast^\T(B,\beta),[1_B])$. 
\end{thm}
\begin{proof} 
It is clear that (1) implies (2).
By \autoref{cor:KKTequivKKeq}, (2) is equivalent to 
the existence of a unital  
$KK$-equivalence $A^\alpha\sim_{KK}B^\beta$.
Since both fixed point algebras are Kirchberg algebras by 
Theorem~6.3
in~\cite{Gar_classificationI_2014}, it follows
from Corollary 4.2.2 in \cite{Phi_classification_2000} that 
$A^\alpha\cong B^\beta$. This shows that (2) implies (3).

We prove that (3) implies (1). 
Denote by $\check{\alpha}$ and $\check{\beta}$ the predual automorphisms of $\alpha$ and $\beta$, respectively, given by Theorem~3.11
in~\cite{Gar_classificationI_2014}. 
Since $A^\alpha\cong B^\beta$, we regard $\check{\alpha}$ and $\check{\beta}$ as automorphisms of the same \ca. 
It follows from \autoref{thm:duality} that 
$\check{\alpha}$ is asymptotically unitarily equivalent
to the identity automorphism, and similarly for $\check{\beta}$.
In particular, $\check{\alpha}$ and $\check{\beta}$
are asymptotically unitarily equivalent.
Thus, the combination of Proposition~6.5 in~\cite{Gar_classificationI_2014} and Theorem~5 in \cite{Nak_aperiodic_2000} implies that
$\check{\alpha}$ and $\check{\beta}$ are cocycle conjugate, 
and thus $\alpha$ and $\beta$ are conjugate by Proposition~2.9 
in~\cite{Gar_classificationI_2014}, as desired.

We turn to the last part of the statement, so 
assume that $A$ and $B$ satisfy the UCT and that there is a 
graded
isomorphism 
$(K_\ast^\T(A,\alpha),[1_A])\cong (K_\ast^\T(B,\beta),[1_B])$. By Julg's theorem (see Theorem~11.7.1 in~\cite{Bla_Ktheory_1998})
in combination with the isomorphism $A\rtimes_\alpha \T\cong A^\alpha
\otimes\K(\ell^2(\Z))$ (and similarly for $\beta)$,
we deduce that $(K_\ast(A^\alpha),[1_{A^\alpha}])\cong (K_\ast(B^\beta),[1_{B^\beta}])$. Since
$A^\alpha$ and $B^\beta$ satisfy the UCT by \autoref{thm: UCTcRp},
it follows that there is a unital $KK$-equivalence 
$A^\alpha\sim_{KK}B^\beta$.
Thus $A^\alpha\otimes C(\T)\sim_{KK^\T} B^\beta\otimes C(\T)$ 
unitally and hence $(A,\alpha)\sim_{KK^\T} (B,\beta)$ by \autoref{prop:KKTequivalenceCTAalpha}. The result follows.
\end{proof}

\autoref{thm: ClassCRp} may be regarded as a \emph{uniqueness} theorem. %It states that two circle actions with the continuous \Rp\ on a unital
%Kirchberg algebra that satisfies the UCT, are conjugate whenever they have the same equivariant $K$-theory.
It is natural to ask for existence results, that is, for a complete
description of the range of the invariant. This is settled in \autoref{thm: rangeInvCRp}. We need some preparation.

\begin{rem}\label{rem:UnitPertAsRepr}
Let $B$ be a unital \ca, let $\varphi\in\Aut(B)$ be an
asymptotically representable automorphism, let $u$ be a unitary
in $B$, and set $\psi=\Ad(u)\circ\varphi$. Then $\psi$ is
asymptotically representable as well by \autoref{thm:duality}, 
since in this setting the 
dual actions $\widehat{\varphi}$ and $\widehat{\psi}$ are conjugate,
and thus $\widehat{\psi}$ has the \cRp.
\end{rem}

\begin{prop}\label{thm:AperOIAsRep}
Any aperiodic automorphism of $\OI$ is asymptotically representable.
%In particular, any Kirchberg algebra admits an aperiodic, asymptotically representable automorphism.
\end{prop}
\begin{proof}
Recall that any two automorphism of $\OI$ are 
asymptotically unitarily equivalent; see
Proposition~2.2.7 in~\cite{Phi_classification_2000}. 
In combination with Theorem~5 in~\cite{Nak_aperiodic_2000}, 
it follows that any two aperiodic automorphisms of $\OI$ are
cocycle conjugate. Since asymptotic representability and
aperiodicity are preserved by inner perturbations (the first 
one by \autoref{rem:UnitPertAsRepr}), it
suffices to find just \emph{one} asymptotically 
representable, aperiodic automorphism of $\OI$. 
Since $\OI$ is isomorphic to $\bigotimes_{n=0}^\I\OI$ by
part~(iii) of~Theorem~7.2.6
in~\cite{RorSto_classification_2002}, it suffices
to find an approximately representable, aperiodic 
automorphism of $\bigotimes_{n=0}^\I\OI$.

%The existence of such an
%automorphism can be deduced, for example, arguing as in Lemma~2.2 in~\cite{BraKisRob_rohlin_2007}.
We provide an explicit construction of such an automorphism.
Recall that $\OI$ has real rank zero 
(see 
Proposition~4.1.1 in~\cite{RorSto_classification_2002}).
For each $n\geq 1$, use Proposition~5.3
in~\cite{RobRor_divisibility_2013} (see also Definition~5.1 there) to 
fix a unital homomorphism $\rho_n\colon {M_n\oplus M_{n+1}}\hookrightarrow \OI$. Set $\widetilde{u}_1=1$ and for $n\geq 2$, 
set \begin{align*} \widetilde{u}_n=\left(
               \begin{array}{cccccc}
                 0 & 1 & 0 & \cdots & 0 & 0 \\
                 0 & 0 & 1 & \cdots & 0 & 0\\
                 0 & 0 & 0 & \ddots & 0 & 0 \\
                 \vdots & \vdots& \vdots & \ddots & \ddots & \vdots \\
                 0 & 0 & 0 & \cdots & 0 & 1 \\
                 1 & 0 & 0 & \cdots & 0 & 0 \\
               \end{array}
             \right) \in M_n.\end{align*}
Set $u_0=1\in \OI$ and for $n\geq 1$, set 
$u_n=\rho_n\left(\widetilde{u}_n, \widetilde{u}_{n+1}\right)\in \OI$. 
Then $\varphi=\bigotimes_{n=0}^\I \Ad(u_n)$ defines an automorphism of
$\bigotimes_{n=0}^\I \OI$. 
In what follows, for $m\geq 0$ we will identify $\bigotimes_{n=0}^m\OI$ canonically with the subalgebra 
$\bigotimes_{n=0}^m\OI\otimes 1\otimes1 \otimes \cdots$
of $\bigotimes_{n=0}^\I \OI$. 

We claim that $\varphi$ is 
asymptotically representable. For $n\geq 1$, let 
$z^{(n)}\colon [0,1]\to M_n$ be a continuos
map with $z^{(n)}_0=1$ and $z^{(n)}_1=\widetilde{u}_n$, 
and such that $[\widetilde{u}_n,z^{(n)}_t]=0$ for all $t\in [0,1]$.
(Such a path exists because the commutant of $\widetilde{u}_n$
in $M_n$ is a finite-dimensional $C^*$-algebra, 
and thus its unitary group is path-connected.) 
We define a continuous unitary path $(v_t)_{t\in [0,\I)}$ 
in $\bigotimes_{n=0}^\I\OI$
as follows. For $k\in\N$ and $t\in [k,k+1]$, we set 
\[v_t=u_0\otimes \cdots \otimes u_k\otimes \rho_{k+1}\big(z^{(k+1)}_{t-k},z^{(k+2)}_{t-k}\big)\]
%\[v_t=\begin{cases*}
%     \rho_1\big(z^{(1)}_t,z^{(2)}_t\big) & if $0\leq t\leq 1$, \\
%     u_1 \otimes \rho_2\big(z^{(2)}_{t-1},z^{(3)}_{t-1}\big) & if $1\leq t\leq 2$,%\\
%     u_1 \otimes u_2\otimes  \rho_3\big(z^{(3)}_{t-2},z^{(4)}_{t-2}\big) & if $2\leq t\leq 3$,\\
%     \vdots\\
%     \end{cases*}\]
which we regard as a unitary in $\bigotimes_{n=0}^\I\OI$ 
in a natural way. 
Then the function $t\mapsto v_t$ is continuous, 
and conditions (1) and (2)
in \autoref{def car} are clearly satisfied. Moreover,
since $\big[v_t,u_0\otimes u_1\otimes \cdots \otimes u_n\big]=0$
for all $t\in  [0,\I)$ and all $n\in\N$, it follows
that $\varphi(v_t)=v_t$ for all $t\in [0,\I)$, and thus
condition (3) in \autoref{def car} is also satisfied. Finally,
to check (4), let $b\in \bigotimes_{n=0}^\I \OI$ be given.
Without loss of generality, we may assume that there
exists $m\in\N$ such that $b$ belongs to $\bigotimes_{n=0}^{m}\OI$, since elements of this form are dense.
Then $\varphi(b)=\Ad(u_0\otimes u_1\otimes\cdots\otimes u_m)(b)$.
For 
$t\geq m+1$, note that $v_t$ has the form 
$v_t=u_0\otimes u_1\otimes \cdots \otimes u_m\otimes \widetilde{v}_t$,
for some unitary $\widetilde{v}_t$ in 
$\bigotimes_{n=m+1}^\I \OI$.
Since $\underbrace{1\otimes \cdots\otimes 1}_{m+1 \mathrm{\ times}} \otimes \widetilde{v}_t$ 
commutes with $b$, we get 
\[v_tbv_t^*= (u_0\otimes u_1\otimes \cdots \otimes u_m) b (u_0\otimes u_1\otimes \cdots \otimes u_m)^*=\varphi(b)\]
whenever $t\geq m+1$, thus establishing (4). The claim follows.

It remains to show that $\varphi$ is aperiodic.
Assume by contradiction 
that there exist $m\in\N$ with $m\geq 1$ and a unitary $z\in \bigotimes_{n=0}^\I \OI$ such that $\varphi^m=\Ad(z)$. Given $\varepsilon<1/2$, find $k\in\N$ and a unitary $w\in \bigotimes_{n=0}^k\mathcal{O}_\infty$ such that $\|z-w\|<\varepsilon$. 
Fix 
$n>\max\{m,k\}$. 
Considering diagonal projections, 
set $\widetilde{p}=e_{1,1}\in M_n$ and $\widetilde{q}=e_{m,m}\in M_n$. Then 
$\widetilde{p}\widetilde{q}=0$ and $\Ad(\widetilde{u}_n^m)(\widetilde{p})=\widetilde{q}$. 
Let $\iota_n\colon \OI\to \bigotimes_{\ell=0}^\I\OI$ denote
the canonical embedding as the $n$-th tensor factor, 
and set $p=\iota_n(\rho_n(\widetilde{p},0))$
and $q=\iota_n(\rho_n(\widetilde{q},0))$. Then
\be\item[(a)] $pq=0$;
\item[(b)] $\varphi^m(p)=q$ (and thus $zpz^*=q$);
\item[(c)] $pw=wp$.
\ee
Using at the third step that $\|z-w\|<\ep$, we get\begin{align*} 1\stackrel{\mathrm{(a)}}{=}\|p-q\|\stackrel{\mathrm{(b)}}{=}\|p-zpz^*\|\leq \|p-wpw^*\|+2\ep\stackrel{\mathrm{(c)}}{=}2\varepsilon<1,\end{align*}
which is a contradiction, so $\varphi^m$ is not inner. It follows that $\varphi$ is aperiodic.
%For the last statement, if $A$ is a Kirchberg algebra and $\varphi\in\Aut(\OI)$ is aperiodic, then one easily shows that 
%$\id_A\otimes \varphi$ determines an aperiodic, asymptotically representable automorphism of $A\otimes\OI\cong A$.
\end{proof}

%\begin{rem}\label{K-theory of model autom cp}
%Let $\psi\in\Aut(\OI)$ be an aperiodic automorphism. Using \autoref{thm:AperOIAsRep} and the Pimsner-Voiculescu exact
%sequence for $\psi$, one immediately gets
%\[K_0(\OI\rtimes_\psi\Z)\cong K_1(\OI\rtimes_\psi\Z)\cong \Z,\]
%with $[1_{\OI\rtimes_\psi\Z}]=1\in K_0(\OI\rtimes_\psi\Z)\cong \Z$. \end{rem}

%\begin{cor}\label{every Kirchberg algebra has a car autom} Let $A$ be a \uca\ such that $A\otimes\OI\cong A$. Then there exists an asymptotically representable, aperiodic automorphism of $A$.\end{cor}
%\begin{proof} Let $\phi\colon A\otimes\OI\to A$ be an isomorphism. Use \autoref{thm:AperOIAsRep} to choose an asymptotically representable, aperiodic automorphism $\psi$ of
%$\OI$. It is then straightforward to show 
%that $\phi\circ\left(\id_A\otimes\psi\right)\circ\phi^{-1}$ is an asymptotically representable automorphism of $A$, and it is
%clearly aperiodic. \end{proof}

The next theorem provides
an existence result for actions on Kirchberg algebras, and it shows 
that the converse to \autoref{cor:KKsymmetric} holds in this setting. 

\begin{thm}\label{thm: rangeInvCRp} Let $A$ be unital Kirchberg algebra.
\be\item There exists an action
of $\T$ on $A$ with the \cRp\ if and only if $A$ is 
nuclearly $KK$-symmetric
(see \autoref{df:KKsymmetric}).
\item For every unital Kirchberg algebra $B$ satisfying 
$A\sim_{KK} C(\T,B)$ unitally, there exists a unique 
(up to conjugacy) circle action 
$\alpha_B\colon \T\to\Aut(A)$ with the continuous
Rokhlin property such that $A^{\alpha_B}\cong B$.
\item Assume that $A$ satisfies the UCT. A triple $(H_0,h_0,H_1)$ consisting of $R(\T)$-modules
$H_0$ and $H_1$, and $h_0\in H_0$, is the equivariant $K$-theory of a circle action on $A$ with the continuous
Rokhlin property, if and only if:
\be
\item there is an isomorphism $K_0(A)\cong K_1(A)$, and
\item there is an isomorphism $\varphi\colon H_0\oplus H_1\to K_0(A)$ with $\varphi(h_0,0)=[1_A]$.
\item the $R(\T)$-module structures on $H_0$ and $H_1$ are trivial.
\ee
Moreover, the circle action on $A$ whose equivariant $K$-theory
is isomorphic to $(H_0,h_0,H_1)$ is unique up to conjugacy.
\ee\end{thm}
\begin{proof} The ``only if'' implication in~(1) is the content
of \autoref{cor:KKsymmetric}. 
Conversely, let $B$ be a nuclear unital \ca\ such that 
$A\sim_{KK}C(\T,B)$ unitally. Since $B$ is unitally 
$KK$-equivalent to a Kirchberg algebra
by \autoref{rem:NucSepKKeqKirch}, 
we may assume that
it is itself a Kirchberg algebra.
We prove the rest of ~(1) 
simultaneously with (2). 

Use \autoref{thm:AperOIAsRep} to find an aperiodic, asymptotically 
representable automorphism $\varphi\in\Aut(\OI)$ and 
set $\psi=\id_B\otimes\varphi\in\Aut(B\otimes\OI)$, 
which we identify with an aperiodic, asymptotically representable automorphism of $B\cong B\otimes\OI$. 
Since $KK(\varphi)=1$, the Pimsner-Voiculescu exact
sequence in $KK$-theory implies 
that there is a unital $KK$-equivalence 
$\OI\rtimes_\varphi\Z
\sim_{KK} C(\T)$. Using this at the last step, 
we get
%crossed product
%by $\varphi$
% satisfying 
\[B\rtimes_\psi\Z\cong B\otimes (\OI\rtimes_\varphi\Z)\sim_{KK}C(\T,B).\]
Hence $A\sim_{KK} B\rtimes_\psi\Z$ unitally, 
and since $B\rtimes_\psi\Z$
is a unital Kirchberg algebra, there exists an isomorphism $B\rtimes_\psi\Z\cong A$.
Denote by $\beta\colon \T\to\Aut(B\rtimes_\psi\Z)$ the dual action of $\varphi$. Then $\beta$ has the continuous Rokhlin
property by \autoref{thm:duality}. 
Let $\alpha_B\colon \T\to\Aut(A)$ denote the action 
induced by $\beta$ under some isomorphism 
$B\rtimes_\psi\Z\cong A$. Then $\alpha_B$ has the continuous 
Rokhlin property, and $A^{\alpha_B}\cong B$. Uniqueness of 
$\alpha_B$ up to conjugacy follows from \autoref{thm: ClassCRp}.

(3) The fact that the equivariant $K$-theory of a circle
action with the \cRp\ satisfies conditions (a) and (b) 
is the content of \autoref{prop:KKTequivalenceCTAalpha}, while
condition (c) follows from \autoref{cor:PredualAsymptInner}. We
prove the converse.

Let $B$ be a unital Kirchberg algebra satisfying the UCT
whose $K$-theory is isomorphic to $(H_0,h_0,H_1)$. By the 
K\"unneth formula, it follows that $C(\T,B)$ has $K$-theory
given by the triple
\[(H_0\oplus H_1, (h_0,0), H_0\oplus H_1).\]
By (a) and (b), we deduce that $A$ and $C(\T,B)$ have 
isomorphic $K$-theories. Since $A$ satisfies the UCT, it follows
that $A$ is unitally $KK$-equivalent to $C(\T,B)$. 
By part~(2) of this theorem, there exists a circle action
on $A$ whose fixed point algebra is isomorphic to $B$. 
In particular, the equivariant $K$-theory of this action
is isomorphic (as groups) to $(H_0,h_0,H_1)$. 
Finally, condition (c) implies that this group isomorphism is
an isomorphism of $R(\T)$-modules, as desired.
\end{proof}

It follows that circle actions on Kirchberg algebras with 
the \cRp\ are ``generated'' by a concrete action. Let
$\varphi\in\Aut(\OI)$ be an aperiodic automorphism, 
and set $C=\OI\rtimes_\varphi\Z$. By Proposition~6.5 in~\cite{Gar_classificationI_2014} and Theorem~5 
in~\cite{Nak_aperiodic_2000}, $C$ does not depend on $\varphi$
and is unitally $KK$-equivalent to $C(\T)$.
Moreover, if $\gamma\colon\T\to\Aut(C)$ denotes the dual 
action of $\varphi$, then $\gamma$ has the \cRp\ by
\autoref{thm:duality}. 

\begin{cor}
Adopt the notation introduced above.
Let $\alpha\colon \T\to\Aut(A)$ be an action on a unital 
Kirchberg algebra $A$. Then $\alpha$ has the \cRp\ if and 
only if $(A,\alpha)$ is conjugate to $(A^\alpha\otimes C,\id_{A^\alpha}\otimes\gamma)$. 
\end{cor}

As an application of \autoref{thm: rangeInvCRp}, we show how 
to compute the number of conjugacy classes of circle actions with
the continuous Rokhlin property that a given Kirchberg algebra has.

\begin{eg} Let $A$ be a unital Kirchberg algebra satisfying the UCT, with %$K$-theory given by
\[K_0(A) \cong  K_1(A) \cong \Z\oplus\Z_6,\]
such that $[1_A]$ corresponds to $(1,0)\in K_0(A)$. By part~(3) of \autoref{thm: rangeInvCRp},
conjugacy classes are in bijection with direct sum decompositions of the form $\Z\oplus\Z_6 \cong H_0\oplus H_1$ that satisfy $(1,0)\mapsto (h_0,0)$ for some $h_0\in H_0$. There are only 4 such direct sum
decompositions, namely:
$$\Z\oplus \Z_6\cong (\Z\oplus \Z_6)\oplus \{0\}\cong (\Z\oplus \Z_2)\oplus \Z_3 \cong (\Z\oplus \Z_3)\oplus \Z_2.$$
(The direct sum decompositions $\{0\}\oplus (\Z\oplus\Z_6)$, $\Z_2\oplus (\Z\oplus\Z_3)$, $\Z_3\oplus (\Z\oplus\Z_2)$ and $\Z_6\oplus \Z$ do not
satisfy condition (3) in \autoref{thm: rangeInvCRp}.) We conclude that there are exactly 4 conjugacy classes.
\end{eg}

We need a small refinement of a well-known result of Kirchberg:

\begin{rem}\label{rem:NucSepKKeqKirch}
Let $A$ be a separable, nuclear \uca. By 
Proposition~8.4.5 in~\cite{RorSto_classification_2002},
there exist a unital Kirchberg algebra $B_0$ and a
$KK$-equivalence $\eta\in KK(A,B_0)$. Let $p\in B_0$
be a projection satisfying 
$[1_A]\times \eta =[p]$, and set $B=pB_0p$. Then
$B$ is a unital Kirchberg algebra (which is $KK$-equivalent
to $B_0$), and $\eta$ induces a unital $KK$-equivalence 
$A\sim_{KK}B$. In particular, every separable, nuclear, \uca\ is \emph{unitally} $KK$-equivalent to a unital Kirchberg algebra.
\end{rem}

As a further application, we prove Theorem~E from the introduction.

\begin{cor}\label{thm:EverycRpKKequivKirch}
Let  
$\alpha\colon \T\to\Aut(A)$ be an action with the \cRp\ on
a separable, nuclear \uca\ $A$. Then there
exist a (unique) unital Kirchberg algebra $D$ and a (unique) 
action $\delta\colon \T\to\Aut(D)$
with the \cRp\ such that $(A,\alpha)\sim_{KK^\T}(D,\delta)$ 
unitally.
\end{cor}
\begin{proof}
Note that $A^\alpha$ is separable and nuclear.
Use \autoref{rem:NucSepKKeqKirch}
to find a unital Kirchberg algebra $B$ which is unitally 
$KK$-equivalent to $A^\alpha$.
Let $D$ be a unital Kirchberg algebra which is unitally
$KK$-equivalent to $C(\T,B)$, and 
use part~(2) of~\autoref{thm: rangeInvCRp} to find an action $\delta\colon \T\to\Aut(D)$ with the \cRp\ satisfying $D^\delta\cong B$. 
Using \autoref{prop:KKTequivalenceCTAalpha} at the first and last step, 
we get the following unital $KK^\T$-equivalences:
\[(A,\alpha)\sim_{KK^\T} (C(\T)\otimes A^\alpha,\texttt{Lt}\otimes\id_{A^\alpha})\sim_{KK^\T} (C(\T)\otimes B,\texttt{Lt}\otimes\id_{B})
 \sim_{KK^\T}(D,\delta).
\]
Uniqueness of $(D,\delta)$ up to conjugacy follows from \autoref{thm: ClassCRp}.
\end{proof}

For actions on Kirchberg algebras, the \cRp\ can be 
completely characterized in terms of $KK$-theory; see \autoref{characterization of cRp}. 
We use this to show that the \Rp\ implies the \cRp\ for
circle actions on Kirchberg algebras with finitely generated $K$-theory,
and find explicit examples that show that they differ in general; see \autoref{eg: not cRp 1} and \autoref{eg: not cRp 2}. 
%For actions on 
%commutative \ca s, we show that there is no difference between the 
%Rokhlin property and its continuous version; see \autoref{comm systsms are tensor products}.

%\subsection{Kirchberg algebras.}
%We characterize those circle actions with the Rokhlin property on Kirchberg algebras that have the continuous \Rp.
%We point out that no UCT assumptions are needed.

\begin{cor}\label{characterization of cRp} 
Let  
$\alpha\colon \T\to\Aut(A)$ be an action with the \Rp\ on
a unital Kirchberg algebra $A$.
Then $\alpha$ has the continuous \Rp\ \ifo $KK(\widehat{\alpha})=[\id_{A\rtimes_\alpha\T}]$.\end{cor}
\begin{proof} If $\alpha$ has the continuous Rokhlin property, then $\widehat{\alpha}$ is asymptotically representable by
\autoref{thm:duality}. Hence it is asymptotically inner, and $KK(\widehat{\alpha})=[\id_{A\rtimes_\alpha\T}]$.

Conversely, assume that 
$KK(\widehat{\alpha})=[\id_{A\rtimes_\alpha\T}]$. 
Let $\varphi\in\Aut(\OI)$ be an aperiodic automorphism, 
and note that it is automatically asymptotically representable
by \autoref{thm:AperOIAsRep}.
By fixing an isomorphism $A\otimes\OI\cong A$, 
we identify $\id_A\otimes \varphi$ with an 
aperiodic and asymptotically representable automorphism
$\psi\in\Aut(A)$. 
Since $\check{\alpha}$ is aperiodic by Proposition~4.5 in~\cite{Gar_classificationI_2014},
it follows from Theorem~5 in~\cite{Nak_aperiodic_2000} that
$\check{\alpha}$ is cocycle conjugate to $\psi$, and in
particular $\check{\alpha}$
is asymptotically representable. \end{proof}

We recall the construction of the $\mbox{PExt}$-group. Given abelian groups $G_1$ and $G_2$, the group $\mbox{PExt}(G_1,G_2)$ is
the subgroup of $\Ext(G_1,G_2)$ consisting of the pure extensions of $G_2$ by $G_1$, that is, those extensions 
$0\to G_1\to G\to G_2\to 0$ such that for every finitely generated subgroup $H_2$ of
$G_2$, if $H$ denotes the preimage of $H_2$ under the canonical quotient map $G\to G_2$, then the induced extension
$$0\to G_1\cap H\to H\to H_2\to 0$$
splits.
See \cite{Sch_pext_2003} for more about the $\mbox{PExt}$-group. We refer the reader to Example 8.4.14 in \cite{RorSto_classification_2002} for the definition
of the $KL$-class of an automorphism.

\begin{cor}\label{cRp and Rp are the same on K algs} 
Let  
$\alpha\colon \T\to\Aut(A)$ be an action with the \Rp\ on
a unital Kirchberg algebra $A$.
Assume that $\mbox{PExt}(K_\ast(A^\alpha),K_{\ast}(SA^\alpha))=0$. Then $\alpha$ has the continuous \Rp. In particular, if $A$ has finitely
generated $K$-theory, then every circle action on $A$ with the \Rp\ has the continuous \Rp. \end{cor}
\begin{proof} Since $\mbox{PExt}(K_\ast(A^\alpha),K_{\ast}(SA^\alpha))=0$, it follows that an automorphism of $A^\alpha$ is $KK$-trivial if and only
if it is $KL$-trivial. Let $\check{\alpha}$ be the predual automorphism of $\alpha$. Then $\check{\alpha}$ is approximately inner by Proposition~3.7 in~\cite{Gar_classificationI_2014}, and in particular $KL$-trivial. 
The first part of the corollary then follows from
\autoref{characterization of cRp} above.

If the $K$-groups of $A$ are finitely generated, then the condition
$$\mbox{PExt}(K_\ast(A^\alpha),K_{\ast}(SA^\alpha))=0$$
is automatically satisfied, since the $K$-groups of $A^\alpha$ are also finitely generated by Theorem~5.5 
in~\cite{Gar_classificationI_2014}. This finishes the proof. \end{proof}

In the corollary above, the condition $\mbox{PExt}(K_\ast(A^\alpha),K_{\ast}(SA^\alpha))=0$ will also be satisfied if
the $K$-groups of $A$ are (possibly infinite) direct sums of cyclic groups. %However, it is in general unclear whether one can
%replace said condition with
%$$\mbox{PExt}(K_\ast(A),K_{\ast+1}(A))=0,$$
%which is significantly easier to check in practice. 

Next, we construct examples of circle actions that have the Rokhlin property but not the
continuous \Rp, showing that these two notions are not equivalent in general. This can happen even on Kirchberg algebras that satisfy
the UCT.
We need to introduce some notation first. Let $A$ be a \uca\ and let $\varphi\in \Aut(A)$ be approximately inner.
With $\iota\colon A\to A\rtimes_\varphi\Z$ denoting the canonical inclusion, the Pimsner-Voiculescu exact sequence for $\varphi$ reduces
to
\begin{equation}\label{eqn:PV}\xymatrix{0\ar[r] & K_j(A)\ar[r]^-{K_j(\iota)} & K_j(A\rtimes_\varphi \Z)\ar[r] & K_{1-j}(A)\ar[r]& 0},\end{equation}
for $j=0,1$. We denote the class of the above extension by $\eta_j(\varphi)$, and by
$$\eta\colon \overline{\mbox{Inn}}(A)\to \mbox{Ext}(K_1(A),K_0(A))\oplus \mbox{Ext}(K_0(A),K_1(A))$$
the map $\eta(\varphi)=(\eta_0(\varphi),\eta_1(\varphi))$ for $\varphi\in \overline{\mbox{Inn}}(A)$. %It is well known that $\eta$ is
%a group homomorphism if $A$ satisfies the UCT, %(the operation on $\overline{\mbox{Inn}}(A)$ is composition) 
%but we will not make
%use of this fact here.

\begin{eg}\label{eg: not cRp 1} Let $G_1=\Z\left[\frac{1}{2}\right]$, regarded as the abelian group generated by elements $t_n$
subject to the relations
$2t_{n+1}=t_n$
for all $n\in \N$. Then $G_1$ is torsion free. Let $G_0=\Z$, and let $E$ be the abelian group generated by the set
$\{x,y_n\colon n\in\N\}$, subject to the relations
$2y_{n+1}=y_n+x$
for all $n\in \N$. There is an extension
$$0\to G_0\to E\to G_1\to 0,$$
where the map $G_0\to E$ is determined by $1\mapsto x$, and the map $E\to G_1$ is the corresponding quotient map. In the next two claims, we will
show that this extension is pure but not trivial (that is, it does not split). 

\textbf{Claim~1:} \emph{$G_0$ is a pure subgroup of $E$.}
For $n\in\N$, let $E_n$ be the subgroup of $E$ generated by $x$ and $y_n$. Then $E_n\cong \Z x\oplus \Z y_n$, and $E_n\subseteq E_{n+1}$ for
all $n\in\N$. Let $H\leq E$ be a finitely generated subgroup containing
$x$. Then there exists $m\in\N$ with $H\subseteq E_m$. Set $H_1=H/G_0$,
which is a subgroup of $\Z y_m$ and is therefore free. It follows that
the extension $0\to G_0 \to H \to H_1\to 0$ splits, proving the claim.

\textbf{Claim~2:} \emph{$G_0$ is not a direct summand in $E$.} Arguing
by contradiction, let $F\leq E$ be a direct complement of $G_0$, 
and note that $F\cong E/G_0$. Let $\pi\colon E\to F$ and $\iota\colon F\to E$ be the canonical quotient map and embedding, respectively. 
For $n\in\N$, set $z_n=\iota(\pi(y_n))\in E$. Then  
$2z_{n+1}=z_n$ for all $n\in\N$, and there exists a unique $k_n\in\Z$
such that $z_n=y_n+k_nx$. Hence 
\[z_n=2z_{n+1}=2y_{n+1}+2k_{n+1}x=y_n+x+2k_{n+1}x=y_n+(2k_{n+1}+1)x.\]
In particular, $k_n=2k_{n+1}-1$ for all $n\in\N$. It follows that
$k_1=2^nk_{n+1}+2^n-1$ and thus $k_1+1$ is divisible by $2^n$ for 
all $n\in\N$. This contradiction proves the claim. 

Denote by $\xi \in
\mbox{Ext}(G_1,G_0)$ the extension class determined by $E$. Note that $\xi\neq 0$.
Use Elliott's classification of A$\T$-algebras (see 
the comments before Proposition 3.2.7 in
\cite{RorSto_classification_2002}),
to find a simple, unital A$\T$-algebra $A$ with real rank zero, such that $K_j(A)\cong G_j$ for $j=0,1$. Use Theorem 3.1 in
\cite{KisKum_class_1998} in the case $i=1$ to find an approximately inner automorphism $\varphi$ of $A$ such that $\eta(\varphi)=(\xi,0)$.
The proof of Theorem 3.1 in \cite{KisKum_class_1998} is constructive, and the case $i=1$ (which is presented in Subsection 3.11 in
\cite{KisKum_class_1998}) shows that for $n\in \N$, there are a circle algebra $A_n$, an embedding $\psi_n\colon A_n\to A_{n+1}$, and a unitary
$u_n\in A_n$, satisfying
$$\Ad(u_{n+1})\circ \psi_n=\psi_n\circ\Ad(u_n)$$
and $\varinjlim \Ad(u_n)=\varphi$. It is immediate to check that such a direct limit action is approximately representable in the sense of
Definition~3.4 in~\cite{Gar_classificationI_2014}. Moreover,
by construction we have $K_0(A\rtimes_\varphi\Z)\cong E$.

Denote by $\alpha\colon\T\to\Aut(A\rtimes_\varphi\Z)$ the dual action of $\varphi$. Then $\alpha$ has the \Rp\ by Proposition 3.6 in
\cite{Gar_classificationI_2014}. On the other hand, since $\eta(\varphi)$ is not the trivial class, we conclude that $\varphi$ is not
asymptotically inner, and hence $\alpha$ does not have the continuous \Rp\ by
\autoref{cor:PredualAsymptInner}.\end{eg}

The example above can be adapted to construct a circle action on a UCT Kirchberg algebra with \Rp\ but not the
continuous \Rp.

\begin{eg} \label{eg: not cRp 2}
Adopt the notation of the previous example. In particular, 
we use the approximately inner 
automorphism $\varphi\in\Aut(A)$ which
satisfies $K_0(A\rtimes_\varphi \Z)\cong E$.

\textbf{Claim:} \emph{$\varphi$ is aperiodic.} 
Arguing by contradiction, suppose that there exist
$n\in\N\setminus\{0\}$ and $w\in\U(A)$ 
such that $\varphi^n=\Ad(w)$. Set $v=w\varphi(w)\cdots \varphi^{n-1}(w)$, which is also a unitary in $A$. One
checks that $\varphi^{n^2}=\Ad(v)$ and that $v$ is $\varphi$-invariant. 
Denote by $D$ the twisted crossed product of $A$ by $\Z$
with respect to the twist induced by $v$.
Using Theorem~2.4 in~\cite{OlePed_partially_1986} at the first
step, and the fact that $\T$ is compact at the second step, 
there exist isomorphisms
\[A\rtimes_\varphi\Z\cong \mathrm{Ind}_{(n^2\Z)^{\perp}}^\T (D)\cong 
 C(\T,D).
\]
In particular, $K_0(A\rtimes_\varphi\Z)\cong 
K_1(A\rtimes_\varphi\Z)\cong E$. Since 
$K_j(A)\cong G_j$ for $j=0,1$, the exact sequence (\ref{eqn:PV}) for $j=1$ gives an extension
\[0\to G_1 \to E \to G_0\to 0.\]
We will show that there is no embedding of $G_1$ into $E$;
for this, it suffices to prove that no element of $E$ is
divisible by 2 infinitely many times. 
Elementary manipulations show that any $e\in E$
has one (and only one) of the following forms:
\be\item[(i)] $e=ax$ for a uniquely determined $a\in\Z$;
\item[(ii)] $e=ax+by_1$ for uniquely determined $a,b\in\Z$;
\item[(iii)] $e=ax+by_m$ for uniquely determined 
$a,b\in\Z$ and $m>1$ with $b$ \emph{odd}.
\ee

In the first case, there exists $f\in E$ with $2f=e$ 
if and only if $a$
is divisible by 2 (and in this case $f=\frac{a}{2}x$). 
In the second case, such an $f$ exists 
if and only if $a-b$ is divisible by 2, and in this case 
\[f=\begin{cases*}
     \frac{a}{2}x+\frac{b}{2}y_1 & if $b$ is even \\
     \frac{a-b}{2}x+by_2 & if $b$ is odd.
    \end{cases*}.\]
Finally, in the third case $f$ exists if and only if 
$a$ is odd and in this case $f=\frac{a-b}{2}x+by_{m+1}$. 
In all cases, the element $e$ is not divisible by $2^k$ for 
\mbox{all} $k\in\N$. Since every element in $G_1$ has this 
property, it follows that there is no embedding 
$G_1\hookrightarrow E$.
This contradiction implies that $\varphi^n$ is not inner, and
thus $\varphi$ is aperiodic.

It follows that the crossed
product $A\rtimes_\varphi\Z$ is simple by Theorem 3.1 in \cite{Kis_outer_1981}, since $A$ is simple and $\varphi$
is aperiodic by the previous claim. 
Set $B=A\otimes\OI$ and $\psi=\varphi\otimes\id_{\OI}$, which is
an automorphism of $B$. 
Since the canonical unital map $\C\to \OI$ is a 
$KK$-equivalence, there is a canonical isomorphism 
$K_\ast(A)\cong K_\ast(B)$. Moreover, this identification 
identifies $\eta(\psi)$ with $\eta(\varphi)$. 
With $\beta\colon\T\to\Aut(B\rtimes_\psi\Z)$ denoting the dual action of $\psi$, the same argument used in \autoref{eg: not cRp 1} shows that
$\beta$ has the \Rp\ and does not have the continuous \Rp. Finally, note that $B\rtimes_\psi\Z\cong (A\rtimes_\varphi\Z)\otimes\OI$ is a
Kirchberg algebra, and it satisfies the UCT because $A$ does, since crossed products by $\Z$ preserve the UCT. \end{eg}

M. Izumi has found \cite{Izu_preparation_2014} examples of $\Z_2$-actions on $\OI$ that are approximately representable (see Definition 3.6
in \cite{Izu_finiteI_2004}) but not asymptotically representable. 
Using duality, one obtains
$\Z_2$-actions on UCT Kirchberg algebras that have the \Rp\ but not the continuous \Rp.
These Kirchberg algebras have infinitely generated $K$-theory, and Izumi shows that his method cannot produce similar examples with
finitely generated $K$-groups. It seems plausible that a result similar to \autoref{cRp and Rp are the same on K algs} holds
for finite groups.% actions as well.

\end{document}